\documentclass[11pt]{amsart}

\usepackage{amsmath,amssymb,mathrsfs,color}
\usepackage{graphicx}
\usepackage{epstopdf}
\usepackage{hyperref}
\hypersetup{hypertex=true,
            colorlinks=true,
            linkcolor=blue,
            anchorcolor=blue,
            citecolor=blue}
\let\cal=\mathcal

\def\R{{\mathbb R}}

\newtheorem{thm}{Theorem}[section]

\newtheorem{lem}[thm]{Lemma}
\newtheorem{prop}[thm]{Proposition}

\theoremstyle{definition}
\newtheorem{de}[thm]{Definition}
\theoremstyle{remark}
\newtheorem{rem}[thm]{Remark}
\newtheorem{exam}[thm]{Example}
\newtheorem{ass}[thm]{Assumption}
\numberwithin{equation}{section}

\newcommand{\rmd}{{\rm d}}
\allowdisplaybreaks

\begin{document}

\title[The locally homeomorphic property of McKean-Vlasov SDEs]{The locally homeomorphic property of McKean-Vlasov SDEs under the global Lipschitz condition}

\author{Xianjin Cheng}
\address{X. Cheng: School of Mathematics, Dalian University of Technology, Dalian 116024, P. R. China}
\email{xjcheng1119@hotmail.com}

\author{Zhenxin Liu}
\address{Z. Liu (Corresponding author): School of Mathematics, Dalian University of Technology, Dalian 116024, P. R. China}
\email{zxliu@dlut.edu.cn}

\thanks{This work is supported by NSFC Grants 11871132, 11925102, and Dalian High-level Talent Innovation Program (Grant 2020RD09).}

\date{June 15, 2023}

\begin{abstract}
In this paper, we establish the locally diffeomorphic property of the solution to McKean-Vlasov stochastic differential equations defined on the Euclidean space. Our approach is built upon the insightful ideas put forth by Kunita. We observe that although the coefficients satisfy the global Lipschitz condition and some suitable regularity condition, the solution in general does not satisfy the globally homeomorphic property at any time except the initial time, which sets McKean-Vlasov stochastic differential equations apart significantly from classical stochastic differential equations. Finally, we provide an example to complement our results.
\end{abstract}

\maketitle

\section{Introduction}
\setcounter{equation}{0}
In order to study the evolution of the density of particles with time, position, and velocity parameters, Kac \cite{KAC} introduced McKean-Vlasov stochastic differential equations (SDEs), also known as the mean-field SDEs. These equations describe the motion of individual particles in a system, where each particle is influenced not only by its own position but also by all particles in the system. This interaction among the particles refers to as the mean field interaction. There are various types of mean-field interactions; see e.g.\ \cite{O}, \cite{MR}. In this paper, we consider the general case: the $d$-dimensional Stratonovich symmetric McKean-Vlasov SDE with the deterministic initial value $x\in\R^d$
\begin{align}\label{isde3}
X_t=x+\sum_{k=0}^{d'}\int_s^t V_k(X_r,P_{X_r})\circ \rmd W_r^k,
\end{align}
where $P_{X_r}$ is the law of $X_r$, $W_t=(W_t^1,...,W^{d'}_t)$ is a $d'$-dimensional Brownian motion, $\circ \rmd W_r^k$ with $k\neq 0$ refers to the Stratonovich symmetric integral, $\circ\rmd W_r^0=\rmd r$ and $V_k$ is a $d$-dimensional function defined on the product space of $\R^d$ and $\cal{P}_2(\R^d)$(the set of probability measures on $\R^d$ with finite second moments). There has been a great deal of works on the existence and uniqueness of solutions and invariant measures for McKean-Vlasov SDEs. For example, Li and Min \cite{LM} obtained the existence and uniqueness of weak solutions when the drift coefficient is bounded measurable and the diffusion coefficient is nondegenerate and Lipschitz continuous. Carmona \cite{C} obtained the existence and uniqueness of strong solutions under Lipschitz conditions. Wang \cite{W} showed the existence and uniqueness of strong/weak solutions and invariant measures under monotone conditions. Dos Reis et al. \cite{F} established existence and uniqueness of strong solutions under random coefficients and drifts of superlinear growth. Mishura and Veretennikov \cite{MV} obtained strong and weak existence/uniqueness of solutions under relaxed regularity conditions. R\"ockner and Zhang \cite{RZ} showed the strong well-posedness of McKean-Vlasov SDEs with the non-degenerate diffusion coefficient under some integrability assumptions in the spatial variable and the Lipschitz continuity in the distribution variable. Liu and Ma \cite{Ma} showed the existence and uniqueness of solutions and exponential ergodicity under the Lyapunov condition, and so on. We denote the unique solution of the above equation by $\Phi_{s,t}(x)$.  Notice that $\Phi_{s,t}(x)$ does not define a flow, as $\Phi_{r,t}(\Phi_{s,r}(x))$ is not the solution of \eqref{isde3} with $\Phi_{s,r}(x)$ as the initial value; see e.g.\ \cite{BLPR} for details. Although the solution to McKean-Vlasov SDEs exists in pairs with its distribution, we can still discuss the homeomorphic property of the solution $\Phi_{s,r}(x)$. The influence of its distribution $P_{\Phi_{s,r}(x)}$ on the homeomorphic property is mainly reflected in the fact that the Jacobian matrix $\partial_x\Phi_{s,t}(x)$ does not satisfy the global invertibility property.

As we know, a classical SDE generates a stochastic flow of diffeomorphisms which has attracted extensive attention. For example, Elworthy \cite{KDE} and Baxendale \cite{PB} converted the classical SDE defined on $\R^d$ into one on some Hilbert manifold consisting of all diffeomorphisms on $\R^d$ and then the diffeomorphism is obtained directly by solving the latter equation. Malliavin \cite{M}, Bismut \cite{B} and Ikeda-Watanabe~\cite{IW} took the flow of the classical SDE as the limit of a sequence of ordinary differential equations with the homeomorphic property, and they demonstrated that the limit also satisfies the homeomorphic property. Meyer~\cite{Me} and Kunita \cite{K1} showed the homeomorphic property through a skillful use of Kolmogorov's continuity criterion. Subsequently, Kunita \cite{K} discovered the inverse mapping of the solution to the classical SDE by utilizing generalized It\^o's formula, leading to the establishment of the homeomorphic property. Moreover, when coefficients of the classical SDE satisfy the global (resp.\ local) Lipschitz condition, the globally (resp.\ locally) homeomorphic property can be guaranteed. The reason why the local Lipschitz condition cannot guarantee the global homeomorphic property is due to the possibility of solution blow-up; see \cite{K1} for more details. However, it should be noted that for McKean-Vlasov SDEs, the globally homeomorphic property does not necessarily hold even under the global Lipschitz condition. This significant distinction between McKean-Vlasov SDEs and classical SDEs stems from the fact that the invertibility of the Jacobian matrix $\partial_x\Phi_{s,t}(x)$ holds only up to a certain stopping time dependent on $x$, rather than at arbitrary times. See Theorem~\ref{thm2} for details.

We will analyze the difficulty overcome in this paper from the following aspect. For classical SDEs, the Jacobian matrix $\partial_x Y_{s,t}(x)$ of the solution $Y_{s,t}(x)$ is globally invertible for any $x$, which makes the inverse mapping $Z_{s,t}$ of $Y_{s,t}$ globally well-defined. However, in the case of McKean-Vlasov SDEs, the invertibility of the Jacobian matrix $\partial_x\Phi_{s,t}(x)$ is only guaranteed up to some stopping time depending on $x$. This leads to that the inverse mapping $\Psi_{s,t}$ of $\Phi_{s,t}$ may not be well-defined, as the equation satisfied by $\Psi_{s,t}(x)$ includes the term $\partial_x\Phi_{s,t}(\Psi_{s,t}(x))^{-1}$. However, we can prove that $\Psi_{s,t}(x)$ is well-defined up to a certain stopping time by means of the truncation technique and that $\Psi_{s,t}$ is indeed an inverse mapping of $\Phi_{s,t}$ by generalized It\^o's formula. To illustrate that the stopping time in general cannot take on the value $+\infty$, even under the global Lipschitz condition, we provide an example. Meanwhile, the example indicates that the globally homeomorphic property cannot be guaranteed at any time except the initial time. Specifically, if the stopping time is independent of $x$, then the globally homeomorphic property can be guaranteed up to the stopping time, but not at arbitrary times.

This paper is organized as follows. In section 2, we review and introduce some preliminaries. Section 3 is dedicated to proving the regularity of the solution with respect to the deterministic initial value. In section 4, we discuss the invertibility of the Jacobian matrix $\partial_x\Phi_{s,t}(x)$ for any $x$, and then prove the locally diffeomorphic property of the solution. Finally, we provide an example to complement our results.

\section{Preliminaries}
\setcounter{equation}{0}
Throughout the paper, we assume that $(\Omega,\cal{F},P)$ is a probability space,  $\R^d$ is the $d$-dimensional Euclidean space
with the Borel $\sigma$-field $\cal{B}(\R^d)$ and $\cal{P}_2(\R^d)$ is the set of all probability measures $\mu$ over $(\R^d,\cal{B}(\R^d))$
with finite second moments, i.e.\ $\int_{\R^d} |x|^2\mu(\rmd x)<\infty$.

\subsection{Differentiation in $\cal{P}_2(\R^d)$\label{2.1}}
In this subsection, we introduce the concept of the Lions derivative and a class of functions defined on $\R^d\times\cal{P}_2(\R^d)$. See
\cite[Section~6]{CP}, \cite{CD1} or \cite{BLPR} for more details. We assume that the probability space $(\Omega, \cal{F}, P)$ is rich enough. ``Rich enough" means that for any $\mu\in\cal{P}_2(\R^d)$ there exists a random variable $v\in L^2(\Omega,\cal{F},P;\R^d)$ such that $P_v=\mu$, i.e.\ $\cal{P}_2(\R^d)=\{P_v;v\in L^2(\Omega,\cal{F},P;\R^d)\}$.
\begin{de}\label{LD}
A function $f:\cal{P}_2(\R^d)\rightarrow\R$ is said to be \emph{Lions differentiable} at $P_{v_0}\in \cal{P}_2(\R^d)$ if $\widetilde{f}$ is Fr\'echet differentiable at $v_0$, where $\widetilde{f}$ is the canonical lift of $f$ from $\cal{P}_2(\R^d)$ to $L^2(\Omega,\cal{F},P;\R^d)$ satisfying $f(P_v)=\widetilde{f}(v)$ for all $v\in L^2(\Omega,\cal{F},P;\R^d)$.
\end{de}
We now explain the reasonability of Definition \ref{LD}. Since $\widetilde{f}$ is Fr\'echet differentiable at $v_0$, there exists a continuous linear functional $D\widetilde{f}(v_0)\in L(L^2(\Omega,\cal{F},P;\R^d);\R)$ such that
\begin{align}\label{T}
f(P_v)-f(P_{v_0})=\widetilde{f}(v)-\widetilde{f}(v_0)=D\widetilde{f}(v_0)(v-v_0)+o(|v-v_0|_{L^2}),
\end{align}
with $|v-v_0|_{L^2}\rightarrow 0$ for $v\in L^2(\Omega,\cal{F},P;\R^d)$. Since $L^2(\Omega,\cal{F},P;\R^d)$ is a Hilbert space, Riesz's representation theorem implies the existence of  a ($P$-$a.s.$) unique random variable $\theta_0\in L^2(\Omega,\cal{F},P;\R^d)$ satisfying $D\widetilde{f}(v_0)(v-v_0)= E[\theta_0\cdot(v-v_0)]$ for all $v\in L^2(\Omega,\cal{F},P;\R^d)$. Moreover, it is shown in \cite[Theorem 6.5]{CP} that there exists a deterministic measurable function $g\in L^{2}_{P_{v_0}}(\R^d,\R^d)$ such that $g(v_0)=\theta_0~P$-$a.s.$, where $g$ depends only on $P_{v_0}$(see \cite[Theorem 6.2]{CP} for details). We call $g=:\partial_{\mu}f(P_{v_0})$ the \textit{Lions derivative} of $f$ at $P_{v_0}$. It is worth noting that the definition of Lions derivatives is independent of the selection of the probability space $(\Omega, \cal{F}, P)$ and the representation $v_0$, which implies the reasonability of Definition~\ref{LD}. Furthermore, the Lions derivative $\partial_{\mu}f(P_{v_0})$ of $f$ at $P_{v_0}$ satisfies
\begin{align*}
f(P_v)-f(P_{v_0})=E[\partial_{\mu} f(P_{v_0},v_0)\cdot(v-v_0)]+o(|v-v_0|_{L^2}),~\hbox{as}~ |v-v_0|_{L^2}\rightarrow 0.
\end{align*}

We call that $f:\cal{P}_2(\R^d)\rightarrow\R$ is Lions differentiable throughout $\cal{P}_2(\R^d)$ if the Fr\'echet derivative of the canonical lift $\widetilde{f}:L^2(\Omega,\cal{F},P;\R^d)\rightarrow\R$ exists at every point in $L^2(\Omega,\cal{F},P;\R^d)$.

For second-order Lions derivatives of $f:\cal{P}_2(\R^d)\rightarrow\R$, we can directly apply the previous discussion to each component of $\partial_{\mu}f(\cdot,y)=\big(\partial^1_{\mu}f(\cdot,y),...,\partial^d_{\mu}f_d(\cdot,y)\big)$ for fixed $y\in\R^d$; see e.g.\ \cite{CM} for details. Similarly, we can define higher-order Lions derivatives. If the $n$-order Lions derivative of $f$ exists, for any multi-index $\alpha=(i_1,...,i_n)\in\{1,...d\}^n$, we denote the $\alpha$-component of the $n$-order Lions derivative by $\partial^\alpha_\mu f=\partial^{i_n}_\mu...\partial^{i_1}_\mu f:\cal{P}_2(\R^d)\times(\R^d)^n\rightarrow\R$.

Although $\partial_{\mu}f(P_v,\cdot)$ is defined $P_v$-$a.s.$, there exists a Lipschitz continuous modification of $\partial_\mu f(P_v ,\cdot)$; see
\cite[Lemma~3.3]{CD1} or \cite{BLPR} for details. Consequently, we can enhance the requirement of differentiability with respect to the second variable under suitable assumptions. It is worth noting that all notions related to Lions derivatives hold in the sense of `modification' in the following text. This enables us to introduce a class of functions defined on $\R^d\times\cal{P}_2(\R^d)$ that satisfy the following assumptions. See \cite{CM} for details.
\begin{ass}\label{con}
We say that $V= (V^1,...,V^d)\in\cal{C}_{b,Lip}^{1,1}(\R^d\times\cal{P}_2(\R^d);\R^d)$ if $\partial_\mu V$ and $\partial_x V$ exist and satisfy: there exists a constant $K>0$ such that
\begin{itemize}
  \item (boundedness) for all $(x,\mu,v)\in \R^d\times\cal{P}_2(\R^d)\times\R^d$,
   $$|\partial_x V(x,\mu)|+|\partial_\mu V(x,\mu,v)|\leq K,$$
  \item (Lipschitz continuity) for all $(x,\mu,v),(x',\mu',v')\in \R^d\times\cal{P}_2(\R^d)\times\R^d$,
    $$|\partial_x V(x,\mu)-\partial_x V(x',\mu')|\leq K(|x-x'|+W_2(\mu,\mu')),$$
    $$|\partial_\mu V(x,\mu,v)-\partial_\mu V(x',\mu',v')|\leq K(|x-x'|+W_2(\mu,\mu')+|v-v'|),$$
    where $W_2$ is the 2-Wasserstein metric.
\end{itemize}
\end{ass}
We say that $V\in\cal{C}_{b,Lip}^{n,n}(\R^d\times\cal{P}_2(\R^d);\R^d)$ if the following conditions are satisfied: for any $i\in\{1,...,d\}$, all multi-indices $\alpha\in\{1,...,d\}^k~(k=0,...,n),\gamma,\beta_1,...,\beta_{\#\alpha}\in\{(i_1,...,i_d);i_j\in\{0,...,n\}\}$ satisfying $\#\alpha+|\beta_1|+...+|\beta_{\#\alpha}|+|\gamma|\leq n$, the derivatives
$$\partial_x^\gamma\partial_{\boldsymbol{v}}^{\boldsymbol{\beta}}\partial_\mu^\alpha V^i(x,\mu,\boldsymbol{v}),
~\partial_{\boldsymbol{v}}^{\boldsymbol{\beta}}\partial_\mu^\alpha\partial_x^\gamma V^i(x,\mu,\boldsymbol{v}),
~\partial_{\boldsymbol{v}}^{\boldsymbol{\beta}}\partial_x^\gamma\partial_\mu^\alpha V^i(x,\mu,\boldsymbol{v})$$
exist and each of these derivatives is bounded and Lipschitz, where $\#\alpha$ is the number of components of $\alpha$, $\boldsymbol{\beta}=(\beta_1,...,\beta_{\#\alpha})$, $\boldsymbol{v}=(v_1,...,v_{\#\alpha})\in(\R^d)^{\#\alpha}$ and $|\cdot|$ represents the sum of its components. When $\#\alpha=0$, we define $\boldsymbol{\beta}=0$. Notice that $\partial_{\boldsymbol{v}}^{\boldsymbol{\beta}}$ and $\partial_\mu^\alpha$ are not exchangeable.

Let us finish the discussion about the Lions derivative by introducing some notations. Let $(\widetilde{\Omega},\cal{\widetilde{F}},\widetilde{P})$ be a copy of the probability space $(\Omega,\cal{F},P)$.  For any random variable $v$ defined on $(\Omega, \cal{F}, P)$, we can denote its copy over $(\widetilde{\Omega},\cal{\widetilde{F}},\widetilde{P})$ by $\widetilde{v}$ satisfying $P_v=\widetilde{P}_{\widetilde{v}}$; see \cite{BLPR} for details. Further, we can extend the notion of copy from random variables to random fields, as a random field can be viewed as a random variable valued in the corresponding function space. Let $\{X_{\lambda},\lambda\in\Lambda\}$ be a random field, where $\Lambda$ is the unit ball of the $e$-dimensional Euclidean space $\R^e$. If $\widetilde{X}$ is a copy of $X$, any finite-dimensional distribution of $X$ coincides with that of $\widetilde{X}$, i.e.\
$P_{(X_{\lambda_1},...,X_{\lambda_n})}=\widetilde{P}_{(\widetilde{X}_{\lambda_1},...,\widetilde{X}_{\lambda_n})}$, for $\lambda_1,...,
\lambda_n\in \Lambda$ and $ n \in \mathbb{N}$.

\subsection{Stratonovich Symmetric McKean-Vlasov SDEs}\label{2.2}
Let $W_t=(W_t^1,...,W^{d'}_t),~t\in\R^+$ be a $d'$-dimensional Brownian motion defined on $(\Omega,\{\cal{F}_t\}_{t\geq0},P)$, where
$\cal{F}_t:=\cal{F}(W_s(0\leq s\leq t)),(t\geq0)$ is a $\sigma$-algebra generated by the history of the Brownian motion up to (and including) time $t$.

We now consider the following Stratonovich symmetric McKean-Vlasov SDE on $\R^d$ with coefficients $V_k~(k=0,...,d')$ belonging to $\cal{C}_{b,Lip}^{n,n}(\R^d\times\cal{P}_2(\R^d);\R^d)~(n\geq2)$
\begin{align}\label{ssde}
\circ \rmd X_t&=\sum_{k=0}^{d'}V_k(X_t,P_{X_t})\circ \rmd W^k_t,
\end{align}
where $\circ \rmd W^k_t$ with $k\neq 0$ represents the Stratonovitch symmetric integral, $\circ\rmd W_t^0=\rmd t$ and $P_{X_t}$ is the law of $X_t$. If there exists a continuous $\R^d$-valued $\cal{F}_t$-adapted process $X_t,~t \geq s$ satisfying the corresponding integral equation
\begin{align*}
X_t=X_s+\sum_{k=0}^{d'}\int_s^tV_k(X_r,P_{X_r})\circ \rmd W^k_r,
\end{align*}
for all $t\geq s$ almost surely, the process $X_t,~t \geq s$ is called \textit{a solution of Stratonovich symmetric McKean-Vlasov SDE} \eqref{ssde} starting from $X_s$ at time $s$.

The solution of the equation is \textit{(path-wise) unique} if two solutions $X_t$ and $X'_t , t \geq s$ of equation~\eqref{ssde} with the same initial condition satisfy $X_t = X'_t$ for all $t \geq s$ almost surely.

Since the Stratonovitch symmetric integral can be expressed by the sum of the corresponding It\^o's integral and some quadratic covariation (see \cite[Proposition 2.4.1]{K}), the existence and the uniqueness of the solution of equation \eqref{ssde} can be reduced to that of the following It\^o McKean-Vlasov SDE
\begin{align}\label{isde}
X_t=X_s+\sum_{k=0}^{d'}\int_s^tV_k(X_r,P_{X_r}) \rmd W^k_r+\int_s^tV'_0(X_r,P_{X_r}) \rmd r,
\end{align}
where $$V'_0(x,\mu)=V_0(x,\mu)+\frac{1}{2}\sum_{k=1}^{d'}\partial_x V_k(x,\mu)V_k(x,\mu).$$
In fact, due to $V_0', V_1,...,V_{d'}$ belonging to $\mathcal{C}_{b,Lip}^{1,1}(\mathbb{R}^d\times\mathcal{P}_2(\mathbb{R}^d),\mathbb{R}^d)$, the existence and uniqueness of the solution of equation \eqref{isde} is guaranteed by \cite[Theorem 2.1]{W}. The unique solution is denoted by $(X_t)_{t\geq s}$ and the martingale part of $V_k(X_t,P_{X_t})$ can be expressed by It\^o's formula (see e.g.\ \cite[Theorem 7.1]{BLPR}) as
\begin{align}\label{ito}
\sum_{k=1}^{d'}\int_s^t\partial_x
V_k(X_r,P_{X_r})V_k(X_r,P_{X_r})\rmd W^k_r.
\end{align}
Therefore, $V_k(X_t,P_{X_t})$ is a continuous semi-martingale, and the quadratic covariation of $V_k(X_t,P_{X_t})$ and $W^k_t$ on the interval $[s, t]$ is computed as
\begin{align*}
\Big\langle V_k(X_t,P_{X_t}),W^k_t\Big\rangle_{s,t}&=\Bigr\langle\int_s^t\partial_x V_k(X_r,P_{X_r})V_k(X_r,P_{X_r})\rmd
W^k_r,W^k_t\Bigr\rangle_{s,t}\\
&=\int_s^t\partial_x V_k(X_r,P_{X_r})V_k(X_r,P_{X_r})\rmd r.
\end{align*}
Combining \cite[Proposition 2.4.1]{K}, this implies that $\int_s^tV_k(X_r,P_{X_r})\circ \rmd W^k_r$ is well defined and satisfies
\begin{align*}
\int_s^tV_k(X_r,P_{X_r})\circ \rmd W^k_r&=\int_s^tV_k(X_r,P_{X_r})\rmd W^k_r+\frac{1}{2}\langle
V_k(X_t,P_{X_t}),W^k_t\rangle_{s,t}\\
&=\int_s^tV_k(X_r,P_{X_r})\rmd W^k_r+\frac{1}{2}\int_s^t\partial_x V_k(X_r,P_{X_r})V_k(X_r,P_{X_r})\rmd r.
\end{align*}
It can be deduced that $(X_t)_{t\geq s}$ is the solution of Stratonovitch symmetric McKean-Vlasov SDE \eqref{ssde}. Conversely, if $(X_t)_{t\geq s}$ is a solution of Stratonovitch symmetric McKean-Vlasov SDE \eqref{ssde}, it is a continuous semi-martingale and satisfies It\^o McKean-Vlasov SDE \eqref{isde}, which implies that the uniqueness of the solution of \eqref{ssde} since equation~\eqref{isde} has a path-wise unique solution.

\subsection{Generalized It\^o's Formula}\label{2.3}

In this subsection, we recall generalized It\^o's formula in the form of Stratonovich symmetric integrals (see \cite[Theorem 2.4.2]{K}), which plays an important role in establishing the locally diffeomorphic property of McKean-Vlasov SDEs. For more details regarding the content of this subsection, see \cite[Chapter 2]{K}.

Let $\mathbb{T}=[0,T]$ and $p>d\vee2$. We call $\phi\in\mathbf{L}^{lip,p}_\mathbb{T}(\R^d)$ if $\phi(x,t)$ is a $d$-dimensional predictable process with spatial parameter $x\in\R^d$ and satisfies
$$\int_{\R^d}\int_0^T|\phi(\omega,x,r)|^2\rmd r\rmd x<\infty,~a.s.,$$
$$\sup_{x\in\R^d}E\int_0^T|\phi(x,r)|^p\rmd r<\infty,$$
\begin{align}\label{lipp}
E\int_0^T|\phi(x,r)-\phi(x',r)|^p\rmd r<C_p|x-x'|^p,
\end{align}
for any $x,x'\in\R^d$ with a positive constant $C_p$.  Moreover, $\phi(x,t)$ is said to belong to $\mathbf{L}^{n+lip,p}_\mathbb{T}(\R^d)$ if it is $n$-times continuously differentiable with respect to $x$ for any $t$ almost surely, and for any $d$-dimensional index $\mathbf{i}$ satisfying $|\mathbf{i}|\leq n$, the derivative $\partial_{x}^{\mathbf{i}}\phi(x,r)$ satisfies \eqref{lipp} by replacing $\phi$ with $\partial_{x}^{\mathbf{i}}\phi$.

Assume that $(f^k(x,t),~k=1,...,d')$ is a $d'$-dimensional predictable processes with spatial parameter $x$ and satisfies
$$f^k(x,t)=\int _s^t(f^k_1(x, r),\rmd W_r)+\int _s^t f^k_2(x,r) \rmd r,$$
where $f^k_1(x, r)$ is a $d\times d'$-matrix-valued random variable and all of its column vectors and $f^k_2(x,r)$ belong to $\mathbf{L}^{3+lip,p}_\mathbb{T}(\R^d)$. Consider the following stochastic process
\begin{align*}
F(x, t) = F(x, s)+\sum_{k=0}^{d'}\int_s^tf^k(x,r) \circ \rmd W^k_r.
\end{align*}
\begin{prop}[Generalized It\^o's formula \cite{K}]\label{GIF}
For any $d$-dimensional continuous semi-martingale $X_t, t\in [0,T]$, the following formula holds
\begin{align}
F(X_t, t) = F(X_s, s) +\sum_{k=0}^{d'}\int_s^tf^k(X_r,r) \circ \rmd W^k_r+\sum_{k=1}^{d'}\int_s^t\frac{\partial F}{\partial x_k}(X_r,r)\circ \rmd X^k_r,
\end{align}
for any $0\leq s < t\leq T$, where $\circ \rmd X^k_r$ with $k\neq0$ represents the Stratonovitch symmetric integral by the continuous semi-martingale $X^k_r$.
\end{prop}

\section{Regularity of Solutions with respect to the Deterministic Initial Value}
In this section, we discuss the regularity of solutions with respect to the deterministic initial value in terms of It\^o McKean-Vlasov SDEs. We will utilize the result in this section to establish analogous conclusions regarding Stratonovich symmetric McKean-Vlasov SDEs in the next section.

For given $T>0$, consider the following It\^o McKean-Vlasov SDE
\begin{align}\label{isde1}
X_t=x+\sum_{k=0}^{d'}\int_s^t V_k(X_r,P_{X_r}) \rmd W_r^k,
\end{align}
where $0\leq s\leq t\leq T$. If $V_0,...,V_d\in \cal{C}_{b,Lip}^{1,1}(\R^d\times\cal{P}_2(\R^d),\R^d)$, the existence and uniqueness of solutions to \eqref{isde1} can be guaranteed. We denote the unique solution by $\Phi_{s,t}(x)$. Under the global Lipschitz condition, Bahlali et al.\ \cite{BM2} investigated the stability of McKean-Vlasov SDEs with respect to the initial value in the sense of $L^2$. Notice that our result here is stronger than that of \cite{BM2}.

\begin{thm}[Continuity]\label{3.1}
Suppose that $V_0,...,V_d$ belong to $\cal{C}_{b,Lip}^{1,1}(\R^d\times\cal{P}_2(\R^d),\R^d)$. Then there exists a modification of the solution $\Phi_{s,t}(x)$ that is continuous with respect to $(s,t,x)$ almost surely.
\end{thm}

\begin{proof}
Through simple calculations, we obtain that for any $0\leq s\leq t\leq T$ and $p\geq2$, there exists a constant $C$ depending on $p,T,d',K$ such that
\begin{align}\label{B}
E|\Phi_{s,t}(x)|^p\leq C(1+|x|^p).
\end{align}
According to Kolmogorov's continuity criterion (see e.g.\ \cite[Theorem 1.8.1]{K} for details), it suffices to verify that for any $p\geq 2$, there exists a positive constant $C'=C(p,T,d',K)$ such that
\begin{align}\label{lp}
E|\Phi_{s,t}(x)-\Phi_{s',t'}(x')|^p\leq C'((1+|x|^p)|s-s'|^{\frac{p}{2}}+(1+|x'|^p)|t-t'|^{\frac{p}{2}}+|x-x'|^p)
\end{align}
for any $(s,t,x)$ and $(s',t',x')$ with $0\leq s < t\leq T$ and $0\leq s'< t'\leq T$. Assume without loss of generality that $s'<s<t<t'$. By the triangle inequality,
\begin{align}\label{e}
&E|\Phi_{s,t}(x)-\Phi_{s',t'}(x')|^p\\
&\leq 3^{p-1}\big(E|\Phi_{s,t}(x)-\Phi_{s',t}(x)|^p+E|\Phi_{s',t}(x)-\Phi_{s',t}(x')|^p+E|\Phi_{s',t}(x')-\Phi_{s',t'}(x')|^p\big).\nonumber
\end{align}

We now estimate the first term of \eqref{e} and note the fact
$$\Phi_{s',t}(x)=\Phi_{s',s}(x)+\sum_{k=0}^{d'}\int_s^tV_k(\Phi_{s',r}(x),P_{\Phi_{s',r}(x)})\rmd W_r^k,$$
for $s'<s<t$. By BDG's inequality, H\"older's inequality, the global Lipschitz condition and \eqref{B}, we obtain
\begin{align*}
E|\Phi_{s,t}(x)-\Phi_{s',t}(x)|^p
\leq C(p,d')E|\Phi_{s',s}(x)-x|^p+C(K,d',p,T)\int_s^tE|\Phi_{s,r}(x)-\Phi_{s',r}(x)|^p\rmd r.
\end{align*}
Subsequently, Gronwall's inequality yields
\begin{align*}
&E|\Phi_{s,t}(x)-\Phi_{s',t}(x)|^p\\
&\leq C'(K,d',p,T)E|\Phi_{s',s}(x)-x|^p\\
&\leq C'(K,d',p,T)C(T,p)|s-s'|^{\frac{p}{2}-1}\sum_{k=0}^{d'}\Big(\int_{s'}^s \big|V_k(0,\delta_0)\big|^p\rmd r\\
&\qquad\qquad\qquad\qquad+\int_{s'}^s E\big|V_k(\Phi_{s',r}(x),P_{\Phi_{s',r}(x)})-V_k(0,\delta_0)\big|^p\rmd
r\Big)\\
&\leq C''(K,d',p,T)|s-s'|^{\frac{p}{2}-1}\Big(|s-s'|+\int_{s'}^sE|\Phi_{s',r}(x)|^p\rmd r\Big)\\
&\leq C'''(K,d',p,T)(1+|x|^p)|s'-s|^{\frac{p}{2}},
\end{align*}
where the third inequality utilizes the global Lipschitz continuity of coefficients and the fourth inequality takes advantage of \eqref{B}.

By a similar approach to the last term in \eqref{e}, we obtain
\begin{align*}
E|\Phi_{s',t}(x')-\Phi_{s',t'}(x')|^p\leq C(d',T,p,K)(1+|x'|^p)\big|t'-t\big|^{\frac{p}{2}}.
\end{align*}

Once again, applying BDG's inequality, H\"older's inequality, the global Lipschitz condition and \eqref{B}, we can estimate the second term of \eqref{e} as
\begin{align*}
&E|\Phi_{s',t}(x)-\Phi_{s',t}(x')|^p\\
&\leq
C(p,d')|x-x'|^p+C(p,d')\sum_{k=0}^{d'}E\Big|\int_{s'}^{t}V_k(\Phi_{s',r}(x),P_{\Phi_{s',r}(x)})-V_k(\Phi_{s',r}(x'),P_{\Phi_{s',r}(x')})
\rmd W_r^k\Big|^p\\
&\leq C(p,d')|x-x'|^p+C(T,p,d')\sum_{k=0}^{d'}\int_{s'}^{t}E\Big|
V_k(\Phi_{s',r}(x),P_{\Phi_{s',r}(x)})-V_k(\Phi_{s',r}(x'),P_{\Phi_{s',r}(x')})\Big|^p\rmd r\\
&\leq C(p,d')|x-x'|^p+C(T,p,K,d')\int_{s'}^{t}E\Big| \Phi_{s',r}(x)-\Phi_{s',r}(x')\Big|^p\rmd r.
\end{align*}
Moreover, it follows from Gronwall's inequality that
$$E|\Phi_{s',t}(x)-\Phi_{s',t}(x')|^p\leq C'(T,p,K,d')|x-x'|^p.$$

Combining the three estimates, \eqref{lp} holds. The proof is complete.
\end{proof}

\begin{rem}\label{LC}
\begin{enumerate}
  \item Meanwhile, \eqref{lp} implies the continuity of the law of $\Phi_{s,t}(x)$ with respect to $(s,t,x)$ in the weak topology of $\cal{P}_2(\R^d)$, which can be proven by the following estimate:
  \begin{align*}
  W_2(P_{\Phi_{s,t}(x)},P_{\Phi_{s',t'}(x')})&\leq\big(E|\Phi_{s,t}(x)-\Phi_{s',t'}(x')|^2\big)^{\frac{1}{2}}\\
  &\leq C_{p}\big(E|\Phi_{s,t}(x)-\Phi_{s',t'}(x')|^p\big)^{\frac{1}{p}}\\
  &\leq C'(T,p,d',K)\big((1+|x|^p)|s-s'|^{\frac{p}{2}}+(1+|x'|^p)|t-t'|^{\frac{p}{2}}+|x-x'|^p\big)^{\frac{1}{p}}\\
  &\leq C''(T,p,d',K)\big((1+|x|)|s-s'|^{\frac{1}{2}}+(1+|x'|)|t-t'|^{\frac{1}{2}}+|x-x'|\big).
  \end{align*}
  \item When applying Kolmogorov's continuity criterion to prove the continuity of the unique solution $Y_{s,t}(x)$ of the classical SDE with respect to $(s,t,x)$, Kunita \cite{K} greatly simplified calculations by the property $Y_{s',t}(x')=Y_{s,t}(Y_{s',s}(x'))$ and the inequality
\begin{align*}
&E[|Y_{s,t}(x)-Y_{s',t}(x')|^p]\\
&=E\bigr[\bigr|Y_{s,t}(x)-Y_{s,t}(Y_{s',s}(x'))\bigr|^p\bigr]\\
&\leq E\Big[E\bigr[|Y_{s,t}(x)-Y_{s,t}(z)\bigr|^p\bigr]|_{z=Y_{s',s}(x')}\Big]
\end{align*}
with $s'<s<t$. However, this approach is not applicable to McKean-Vlasov SDEs because of the lack of the flow property. Fortunately, we can estimate the $p$-moment $E[|\Phi_{s,t}(x)-\Phi_{s',t}(x')|^p]$ using
$$\Phi_{s',t}(x')=\Phi_{s',s}(x')+\sum_{k=0}^{d'}\int_s^tV_k(\Phi_{s',r}(x'),P_{\Phi_{s',r}(x')})\rmd W_r^k.$$
Moreover, a similar situation arises regarding the continuity of the derivatives with respect to $(s,t,x)$.
\end{enumerate}

\end{rem}

We will next show the existence of the derivatives of all order of $\Phi_{s,t}(x)$ with respect to $x$ and the continuity of  $\partial^{\mathbf{i}}_x\Phi_{s,t}(x)$ with respect to $(s,t,x)$ for any $d$-dimensional index $\mathbf{i}$.

\begin{thm}\label{2}
Suppose that $V_0,...,V_d$  belong to $\cal{C}_{b,Lip}^{1,1}(\R^d\times\cal{P}_2(\R^d),\R^d)$ and $\Phi_{s,t}(x)$ is the unique solution of
It\^o McKean-Vlasov SDE \eqref{isde1} starting from $x$ at $s$. Then there exists a modification of this solution, still denoted by $\Phi_{s,t}(x)$,  such that for all $s,t\in [0, T]$, the mapping $x\mapsto\Phi_{s,t}(x)$ is $P$-$a.s.$ differentiable, and its derivative $\partial_x\Phi_{s,t}(x)=\Big(\frac{\partial\Phi_{s,t}}{\partial x_i}(x)\Big)$ is continuous with respect to $(s,t,x)$ almost surely and satisfies the following SDE
\begin{align*}
\partial_x\Phi_{s,t}(x)=I+\sum_{k=0}^{d'}&\int_s^t \partial_x V_k(\Phi_{s,r}(x),P_{\Phi_{s,r}(x)})\cdot \partial_x\Phi_{s,r}(x)\\
&+\widetilde{E}\bigr[\partial_{\mu}V_k(\Phi_{s,r}(x),P_{\Phi_{s,r}(x)},\widetilde{\Phi}_{s,r}(x))\cdot
\partial_x\widetilde{\Phi}_{s,r}(x)\bigr]  \rmd W_r^k,
\end{align*}
where $I=I_{d\times d}$ is the unit matrix. Further, if $V_0,...,V_d$ belong to $\cal{C}_{b,Lip}^{n,n}(\R^d\times\cal{P}_2(\R^d),\R^d)$, there exists a modification of $\Phi_{s,t}(x)$ such that it is $n$-times continuously differentiable with respect to $x$ and the derivatives of all order are continuous in $(s,t,x)$ almost surely.
\end{thm}

\begin{proof}
{\bf Step 1, we will demonstrate that the mapping $x\mapsto\Phi_{s,t}(x)$ is $P$-$a.s.$ differentiable.} Set
$$\eta_{s,t}^{\epsilon,l}(x)=\frac{1}{\epsilon}\{\Phi_{s,t}(x+\epsilon e_l)-\Phi_{s,t}(x)\},$$ which satisfies
\begin{align*}
\eta_{s,t}^{\epsilon,l}(x)=e_l + \sum_{k=0}^{d'}&\int_s^t\frac{1}{\epsilon}\cdot[V_k(\Phi_{s,r}(x + \epsilon e_l),P_{\Phi_{s,r}(x +
\epsilon e_l)}) - V_k(\Phi_{s,r}(x),P_{\Phi_{s,r}(x)})]\rmd W_r^k\\
=e_l + \sum_{k=0}^{d'}\int_s^t\int_0^1&\frac {d}{\rmd \tau}\big[V_k(\Phi_{s,r}(x) + \tau (\Phi_{s,r}(x + \epsilon e_l) -
\Phi_{s,r}(x)),P_{\Phi_{s,r}(x) + \tau (\Phi_{s,r}(x + \epsilon e_l) - \Phi_{s,r}(x))})\big]\rmd \tau \frac{1}{\epsilon}\rmd W_r^k \\
=e_l + \sum_{k=0}^{d'}\int_s^t\int_0^1&\Big[\partial_x V_k(\Phi_{s,r}(x) + \tau (\Phi_{s,r}(x + \epsilon e_l) - \Phi_{s,r}(x)),P_{\Phi_{s,r}(x)
+ \tau (\Phi_{s,r}(x + \epsilon e_l) - \Phi_{s,r}(x))})\cdot \eta_{s,r}^{\epsilon,l}(x)\Big]\rmd \tau \rmd W_r^k\\
+\sum_{k=0}^{d'}\int_s^t\int_0^1&\widetilde{E}\Big[\partial_\mu V_k\big(\Phi_{s,r}(x) + \tau(\Phi_{s,r}(x + \epsilon e_l) -
\Phi_{s,r}(x)),P_{\Phi_{s,r}(x) + \tau (\Phi_{s,r}(x + \epsilon e_l) - \Phi_{s,r}(x))},\\
&\qquad\qquad\qquad\qquad\qquad\qquad\widetilde{\Phi}_{s,r}(x) + \tau (\widetilde{\Phi}_{s,r}(x + \epsilon e_l) - \widetilde{\Phi}_{s,r}(x))\bigr)\cdot
\widetilde{\eta}_{s,r}^{\epsilon,l}(x)\Big]\rmd \tau \rmd W_r^k.
\end{align*}

We first prove the boundedness of $E\|\eta_{s,\cdot}^{\epsilon,l}(x)\|_t^p$ for $p\geq2$ and $\|X_{\cdot}\|_t=\sup_{r\in[s,t]}|X_r|$. By the martingale inequality, BDG's inequality, H\"older's inequality and the boundedness in Assumption \ref{con}, we obtain
\begin{align*}
&E\|\eta_{s,\cdot}^{\epsilon,l}(x)\|_t^p\\
&\leq C(p,d')\Bigr(1+\sum_{k=0}^{d'}\Big(E\max_{s\leq u\leq t}\Big|\int_s^u\int_0^1\partial_x V_k\cdot \eta_{s,r}^{\epsilon,l}(x)\rmd \tau \rmd W_r^k\Big|^p
+E\max_{s\leq u\leq t}\Big|\int_s^u\int_0^1\widetilde{E}[\partial_\mu V_k\cdot\widetilde{\eta}_{s,r}^{\epsilon,l}(x)]\rmd \tau \rmd W_r^k\Big|^p\Big)\Bigr)\\
&\leq C(p,d')+C'(p,d',T)\sum_{k=0}^{d'}\Big(E\int_s^t\int_0^1|\partial_x V_k\cdot\eta_{s,r}^{\epsilon,l}(x)|^p+|\widetilde{E}[\partial_\mu V_k\cdot \widetilde{\eta}_{s,r}^{\epsilon,l}(x)]|^p\rmd\tau\rmd r\Big)\\
&\leq C(p,d')+C'(p,d',T,K)\Big(E\int_s^t\int_0^1|\eta_{s,r}^{\epsilon,l}(x)|^p+\widetilde{E}|\widetilde{\eta}_{s,r}^{\epsilon,l}(x)|^p\rmd \tau\rmd r\Big)\\
&\leq C(p,d')+C'(p,T,d',K)\int_s^tE\|\eta_{s,\cdot}^{\epsilon,l}(x)\|_r^p\rmd r.
\end{align*}
Gronwall's inequality yields
\begin{align}\label{bdd}
E\|\eta_{s,\cdot}^{\epsilon,l}(x)\|_t^p\leq \widetilde{C}
\end{align}
for a suitable constant $\widetilde{C}$ independent of $s,t,x,\epsilon,l$.

For ease of computation, we set
\begin{align*}
\alpha_k(s,r,\tau,x,\epsilon):=&\partial_x V_k(\Phi_{s,r}(x)+\tau \epsilon\eta_{s,r}^{\epsilon,l}(x),P_{\Phi_{s,r}(x)+\tau \epsilon\eta_{s,r}^{\epsilon,l}(x)}),\\
\beta_k(s,r,\tau,x,\epsilon):=&\partial_\mu V_k(\Phi_{s,r}(x)+\tau \epsilon\eta_{s,r}^{\epsilon,l}(x),P_{\Phi_{s,r}(x)+\tau
\epsilon\eta_{s,r}^{\epsilon,l}(x)},\widetilde{\Phi}_{s,r}(x)+\tau \epsilon\widetilde{\eta}_{s,r}^{\epsilon,l}(x)).
\end{align*}
Notice that the two terms can be bounded by the constant $K$.
Now, we will calculate the $p$-moment of $\eta_{s,t}^{\epsilon,l}(x)-\eta_{s,t}^{\epsilon',l}(x')$. Note that $\eta_{s,t}^{\epsilon,l}(x)-\eta_{s,t}^{\epsilon',l}(x')$ is the sum of the following terms
\begin{align*}
&\sum_{k=0}^{d'}\int_s^t\int_0^1\alpha_k(s,r,\tau,x,\epsilon)\cdot \eta_{s,r}^{\epsilon,l}(x)-\alpha_k(s,r,\tau,x',\epsilon')\cdot
\eta_{s,r}^{\epsilon',l}(x')\rmd \tau \rmd W_r^k=:\sum_{k=0}^{d'}\int_s^t(*_1) \rmd W_r^k,
\end{align*}
\begin{align*}
&\sum_{k=0}^{d'}\int_s^t\int_0^1\widetilde{E}[\beta_k(s,r,\tau,x,\epsilon)\cdot \widetilde{\eta}_{s,r}^{\epsilon,l}(x)]-\widetilde{E}[\beta_k(s,r,\tau,x',\epsilon')\cdot \widetilde{\eta}_{s,r}^{\epsilon',l}(x')]\rmd \tau\rmd W_r^k=:\sum_{k=0}^{d'}\int_s^t(*_2) \rmd W_r^k.
\end{align*}
We first deal with the first term. By Assumption \ref{con}, we deduce
\begin{align*}
&|(*_1)|=\Big|\int_0^1\alpha_k(s,r,\tau,x,\epsilon)\cdot \eta_{s,r}^{\epsilon,l}(x)-\alpha_k(s,r,\tau,x',\epsilon')\cdot
\eta_{s,r}^{\epsilon',l}(x')\rmd \tau\Big|\\
&\leq\int_0^1|\alpha_k(s,r,\tau,x,\epsilon)\cdot(\eta_{s,r}^{\epsilon,l}(x)-\eta_{s,r}^{\epsilon',l}(x'))|+|\alpha_k(s,r,\tau,x,\epsilon)
-\alpha_k(s,r,\tau,x',\epsilon')||\eta_{s,r}^{\epsilon',l}(x')|\rmd \tau\\
&\leq K|\eta_{s,r}^{\epsilon,l}(x)-\eta_{s,r}^{\epsilon',l}(x')|+K|\eta_{s,r}^{\epsilon',l}(x')|\\
&\qquad\qquad\qquad\times \int_0^1|(1-\tau)(\Phi_{s,r}(x)-\Phi_{s,r}(x'))+\tau( \Phi_{s,r}(x+\epsilon e_l)-\Phi_{s,r}(x'+\epsilon's e_l))|\\
&\qquad\qquad\qquad\qquad\qquad\qquad\qquad\qquad\qquad+W_2(P_{\Phi_{s,r}(x)+\tau\epsilon\eta_{s,r}^{\epsilon,l}(x)}, P_{\Phi_{s,r}(x')+\tau \epsilon'\eta_{s,r}^{\epsilon',l}(x')})\rmd\tau\\
&\leq K|\eta_{s,r}^{\epsilon,l}(x)-\eta_{s,r}^{\epsilon',l}(x')|+K|\eta_{s,r}^{\epsilon',l}(x')|\\
&\qquad\qquad\qquad\times\Bigr(|\Phi_{s,r}(x)-\Phi_{s,r}(x')|+|\Phi_{s,r}(x+
\epsilon e_l)-\Phi_{s,r}(x'+\epsilon's e_l)|\\
&\qquad\qquad\qquad\qquad+\bigr(E|\Phi_{s,r}(x)-\Phi_{s,r}(x')|^2\bigr)^{\frac{1}{2}}+\bigr(E|\Phi_{s,r}(x+\epsilon e_l)-\Phi_{s,r}(x'+\epsilon'e_l)|^2\bigr)^{\frac{1}{2}}\Bigr).
\end{align*}
Then it follows from H\"older's inequality and \eqref{bdd} that
\begin{align*}
E|(*_1)|^p&\leq C(p,K,\widetilde{C})
E|\eta_{s,r}^{\epsilon,l}(x)-\eta_{s,r}^{\epsilon',l}(x')|^p+C(p,K,\widetilde{C})\Bigr(E|\Phi_{s,r}(x)-\Phi_{s,r}(x')|^{2p}\Bigr)^{\frac{1}{2}}\\
&\qquad\qquad+C(p,K,\widetilde{C})\Bigr(E|\Phi_{s,r}(x+\epsilon e_l)-\Phi_{s,r}(x'+\epsilon's e_l)|^{2p}\Bigr)^{\frac{1}{2}}.
\end{align*}
Since $E\bigr[|\Phi_{s,r}(x)-\Phi_{s,r}(x')|^{2p}\bigr]\leq C|x-x'|^{2p}$, we have
\begin{align*}
&E\Big|\sum_{k=0}^{d'}\int_{s}^t(*_1)\rmd W_k^r\Big|^p\\
&\leq C(p,d',K,T)\int_{s}^t E|(*_1)|^p \rmd r\\
&\leq C(p,d',K,T)\int_{s}^tE|\eta_{s,r}^{\epsilon,l}(x)-\eta_{s,r}^{\epsilon',l}(x')|^p+\Bigr(E|\Phi_{s,r}(x)-\Phi_{s,r}(x')|^{2p}\Bigr)^{\frac{1}{2}}\\
&\qquad\qquad\qquad\qquad+\Bigr(E|\Phi_{s,r}(x+\epsilon e_l)-\Phi_{s,r}(x'+\epsilon'e_l)|^{2p}\Bigr)^{\frac{1}{2}}\rmd r\\
&\leq C(p,T,d',K)\int_{s}^tE\|\eta_{s,\cdot}^{\epsilon,l}(x)-\eta_{s,\cdot}^{\epsilon',l}(x')\|_r^p\rmd r+C(p,T,d',K)\Big(|x-x'|^p+|\epsilon-\epsilon'|^p\Big).
\end{align*}
After same calculations on $(*_2)$, we obtain
\begin{align*}
&E\Big|\sum_{k=0}^{d'}\int_{s}^t(*_2)\rmd W_k^r\Big|^p\\
&\leq C(p,T,d',K)\int_{s}^tE\|\eta_{s,\cdot}^{\epsilon,l}(x)-\eta_{s,\cdot}^{\epsilon',l}(x')\|_r^p\rmd r+C(p,T,d',K)\Big(|x-x'|^p+|\epsilon-\epsilon'|^p\Big).
\end{align*}
Therefore,  it follows that
\begin{align*}
&E\|\eta_{s,\cdot}^{\epsilon,l}(x)-\eta_{s,\cdot}^{\epsilon',l}(x')\|_t^p\\
&\leq C(p,d')\sum_{k=0}^{d'}E\Big|\int_{s}^t(*_1)+(*_2)\rmd W_k^r\Big|^p\\
&\leq C(p,d',T,K)\int_{s}^tE\|\eta_{s,\cdot}^{\epsilon,l}(x)-\eta_{s,\cdot}^{\epsilon',l}(x')\|_r^p\rmd r+C(p,d',T,K)\Big(|x-x'|^p+|\epsilon-\epsilon'|^p\Big).
\end{align*}
Gronwall's inequality implies that, for any $s\geq0$ and $p\geq2$, there exists a positive constant $C$ depending on $p,T,d',K$ such that for any $x,x'\in \R^d$ and $0<|\epsilon|,|\epsilon'|<1$,
\begin{align}\label{lip}
E\|\eta_{s,\cdot}^{\epsilon,l}(x)-\eta_{s,\cdot}^{\epsilon',l}(x')\|_T^p\leq C\Big(|x-x'|^p+|\epsilon-\epsilon'|^p\Big).
\end{align}
Moreover, we have
\begin{align}\label{lipw}
\widetilde{E}|\widetilde{\eta}_{s,t}^{\epsilon,l}(x)-\widetilde{\eta}_{s,t}^{\epsilon',l}(x')|^p\leq C\Bigr(|x-x'|^p+|\epsilon-\epsilon'|^p\Bigr)
\end{align}
for any $t\in[s,T].$

Set $$B:=\{X_{\cdot}|\R^d\hbox{-valued continuous function of}~t\in [s,T]\},$$
which is a Banach space with the norm $\|X_{\cdot}\|=\max_{t\in [s,T]}|X_t|$. Then $\eta_{s,\cdot}^{\epsilon,l}(x)$ can be viewed as a stochastic process valued in $B$. According to Kolmogorov's continuity criterion and \eqref{lip}, it holds that for any $t>s$, $\eta_{s,t}^{\epsilon,l}(x)$ has a continuous version at $\epsilon=0$, still denoted by $\eta_{s,t}^{\epsilon,l}(x)$.  Moreover, $\lim_{\epsilon\rightarrow0}\eta_{s,t}^{\epsilon,l}(x)$ exists uniformly with respect to $t$, denoted by $\partial_l\Phi_{s,t}(x)$, which is continuous in $x$. This implies that for $s < t$, $\Phi_{s,t}(x)$ is continuously differentiable with respect to $x$ almost surely. Similarly, for any $t>s$, $\widetilde{\eta}_{s,t}^{\epsilon,l}(x)$ has a continuous version at $\epsilon=0$ and
$\lim_{\epsilon\rightarrow0}\widetilde{\eta}_{s,t}^{\epsilon,l}(x)=:\partial_l\widetilde{\Phi}_{s,t}(x)$ exists. The proof of Step 1 is complete. Furthermore, we can obtain the following conclusion.

 Since both $\widetilde{\eta}_{s,t}^{\epsilon,l}(x)$ and $\eta_{s,t}^{\epsilon,l}(x)$ are continuous at $\epsilon=0$ and $\widetilde{\Phi}_{s,t}(x)$ is a copy of $\Phi_{s,t}(x)$ in the sense of `random fields' (see Subsection~\ref{2.1}),
\begin{align}\label{p}
E|\partial_l\Phi_{s,t}(x)-\partial_l\Phi_{s,t}(x')|^p=\widetilde{E}|\partial_l\widetilde{\Phi}_{s,t}(x)-\partial_l\widetilde{\Phi}_{s,t}(x')|^p
\end{align}
holds for any $s<t$ and $x,x'\in\R^d$. By \eqref{bdd}, we obtain
\begin{align}\label{bdd1}
\sup_{t\in [s,T]}E|\partial_l\Phi_{s,t}(x)|^P\leq E\|\partial_l\Phi_{s,\cdot}(x)\|_T^P\leq \widetilde{C}.
\end{align}
By \eqref{lip}, it follows that $\partial_l\Phi_{s,t}(x)$ is Lipschitz, i.e.
\begin{align}\label{lip1}
\sup_{t\in [s,T]}E|\partial_l\Phi_{s,t}(x)-\partial_l\Phi_{s,t}(x')|^p\leq E\|\partial_l\Phi_{s,\cdot}(x)-\partial_l\Phi_{s,\cdot}(x')\|_T^p\leq
C|x-x'|^p.
\end{align}

{\bf Step 2: We are going to search for It\^o's SDE satisfied by $\partial_x\Phi_{s,t}(x)$.}
For given $M\in \mathbb{N}$, let $B(0,M)$ denote the open ball centered at $0$ with radius $M$. To derive the equation satisfied by $\partial_l\Phi_{s,t}(x)$, it suffices to verify that $V_k(\Phi_{s,\cdot}(\cdot),P_{\Phi_{s,\cdot}(\cdot)})\in \mathbf{L}^{1+Lip,p}_{[s,T]}(B(0,M))$, which guarantees the interchangeability of It\^o's integral and the derivative $\partial_l$ according to \cite[Proposition 2.3.1]{K}. More precisely, we need to demonstrate that there exists a constant $C>0$ such that for any $x,x'\in B(0,M)$,
$$\int_{B(0,M)}\int_s^T|V_k(\Phi_{s,r}(x),P_{\Phi_{s,r}(x)})|^2\rmd r\rmd x<\infty,~a.s.,$$
$$\sup_{x\in B(0,M)}E\int_s^T|V_k(\Phi_{s,r}(x),P_{\Phi_{s,r}(x)})|^p\rmd r<\infty,$$
\begin{align}\label{111}
E\int_s^T\big|\partial_x^{\textbf{i}}\big[V_k(\Phi_{s,r}(x),P_{\Phi_{s,r}(x)})\big]-\partial_x^{\textbf{i}}\big[V_k(\Phi_{s,r}(x'),P_{\Phi_{s,r}(x')})\big]\big|^p\rmd r&\leq C|x-x'|^p,
\end{align}
where $\textbf{i}$ is the $d$-dimensional index with $|\textbf{i}|\leq1.$ In fact, the first inequality holds because of the Lipschitz continuity of $V_k$, Theorem \ref{3.1} and Remark \ref{LC}. Meanwhile, we obtain the second inequality by \eqref{B} and the triangle inequality with the addition of $V_k(0,\delta_0)$. The Lipschitz continuity of $V_k$ and estimate~\eqref{lp} imply \eqref{111} when $\textbf{i}$ is the zero vector.

 For $\textbf{i}=e_i~(i=1,...,d)$,
\begin{align*}
&\partial_x^{\textbf{i}}\big[V_k(\Phi_{s,r}(x),P_{\Phi_{s,r}(x)})\big]\\
&=\partial_x V_k(\Phi_{s,r}(x),P_{\Phi_{s,r}(x)})\frac{\partial\Phi_{s,r}}{\partial x_i}(x)
+\widetilde{E}[\partial_{\mu}V_k(\Phi_{s,r}(x),P_{\Phi_{s,r}(x)},\widetilde{\Phi}_{s,r}(x))
\frac{\partial\widetilde{\Phi}_{s,r}}{\partial x_i}(x)].
\end{align*}
For ease of computation, we define
\begin{align*}
a_k(s,r,x)=&\partial_x V_k(\Phi_{s,r}(x),P_{\Phi_{s,r}(x)}),\\
b_k(s,r,x)=&\partial_\mu V_k(\Phi_{s,r}(x),P_{\Phi_{s,r}(x)},\widetilde{\Phi}_{s,r}(x)).
\end{align*}
After some calculations we have
\begin{align*}
&E\int_s^T\big|\partial_x^{\textbf{i}}\big[V_k(\Phi_{s,r}(x),P_{\Phi_{s,r}(x)})\big]-\partial_x^{\textbf{i}}\big[V_k(\Phi_{s,r}(x'),P_{\Phi_{s,r}(x')})\big]\big|^p\rmd r\\
&\leq 2^pE\int_s^T\Big|a_k(s,r,x)\frac{\partial\Phi_{s,r}}{\partial x_i}(x)-a_k(s,r,x')\frac{\partial\Phi_{s,r}}{\partial x_i}(x')\Big|^p\rmd r\\
&\quad+2^pE\int_s^T\Big|\widetilde{E}\Big[b_k(s,r,x)\frac{\partial\widetilde{\Phi}_{s,r}}{\partial x_i}(x)\Big]-
\widetilde{E}\Big[b_k(s,r,x')\frac{\partial\widetilde{\Phi}_{s,r}}{\partial x_i}(x')\Big]\Big|^p\rmd r.
\end{align*}
Since the second term is similar to the first term, we only estimate the first term. It follows from Assumption \ref{con}, \eqref{lp1} and \eqref{BB} that
\begin{align*}
&E\int_s^T\Big|a_k(s,r,x)\frac{\partial\Phi_{s,r}}{\partial x_i}(x)-a_k(s,r,x')\frac{\partial\Phi_{s,r}}{\partial x_i}(x')\Big|^p\rmd r\\
&\leq 2^pE\int_s^T|a_k(s,r,x)|^p\Big|\frac{\partial\Phi_{s,r}}{\partial x_i}(x)-\frac{\partial\Phi_{s,r}}{\partial x_i}(x')\Big|^p+|a_k(s,r,x)-a_k(s,r,x')|^p\Big|\frac{\partial\Phi_{s,r}}{\partial x_i}(x')\Big|^p\rmd r\\
&\leq 2^p\int_s^TK^p\Big(E\Big|\frac{\partial\Phi_{s,r}}{\partial x_i}(x)-\frac{\partial\Phi_{s,r}}{\partial x_i}(x')\Big|^{2p}\Big)^{\frac{1}{2}}+\Big(E|a_k(s,r,x)-a_k(s,r,x')|^{2p}\Big)^{\frac{1}{2}}\Big(E\Big|\frac{\partial\Phi_{s,r}}{\partial x_i}(x')\Big|^{2p}\Big)^{\frac{1}{2}}\rmd r\\
&\leq C(p,T,K,d')\int_s^T\Big(E\Big|\frac{\partial\Phi_{s,r}}{\partial x_i}(x)-\frac{\partial\Phi_{s,r}}{\partial x_i}(x')\Big|^{2p}\Big)^{\frac{1}{2}}+\Big(E|a_k(s,r,x)-a_k(s,r,x')|^{2p}\Big)^{\frac{1}{2}}\rmd r\\
&\leq C'(p,T,K,d')\int_s^T\Big(E\Big|\frac{\partial\Phi_{s,r}}{\partial x_i}(x)-\frac{\partial\Phi_{s,r}}{\partial x_i}(x')\Big|^{2p}\Big)^{\frac{1}{2}}+\Big(E\big|\partial_x\Phi_{s,r}(x)-\partial_x\Phi_{s,r}(x')\big|^{2p}\Big)^{\frac{1}{2}}\rmd r\\
&\leq C''(p,T,K,d')|x-x'|^p.
\end{align*}
This implies that \eqref{111} holds for $\textbf{i}=e_i~(i=1,...,d)$. So far, we have proven that $V_k(\Phi_{s,\cdot}(\cdot),P_{\Phi_{s,\cdot}(\cdot)})$ belongs to $\mathbf{L}^{1+Lip,p}_{[s,T]}(B(0,M))$ with $k\neq0$.

When $k=0$, the interchangeability of Riemann integrals and the derivative $\partial_l$ can be guaranteed under the condition $V_0\in\cal{C}_{b,Lip}^{1,1}(\R^d\times\cal{P}_2(\R^d),\R^d)$.

Consequently, $\partial_x\Phi_{s,t}(x)=(\partial_1\Phi_{s,t}(x),...,\partial_d\Phi_{s,t}(x))$ satisfies
\begin{align*}
\partial_x\Phi_{s,t}(x)=I+\sum_{k=0}^{d'}&\int_s^t \partial_x V_k(\Phi_{s,r}(x),P_{\Phi_{s,r}(x)})\cdot \partial_x\Phi_{s,r}(x)\\
&+\widetilde{E}\bigr[\partial_{\mu}V_k(\Phi_{s,r}(x),P_{\Phi_{s,r}(x)},\widetilde{\Phi}_{s,r}(x))\cdot
\partial_x\widetilde{\Phi}_{s,r}(x)\bigr]  \rmd W_r^k.
\end{align*}

{\bf Step 3, we will prove that $\partial_x\Phi_{s,t}(x)$ is continuous in $(s,t,x)$ almost surely.} The proof is the same as that of Theorem 3.1. By a simple calculation, we obtain
\begin{align}\label{BB}
E|\partial_x\Phi_{s,t}(x)|^p\leq C'
\end{align}
for a appropriate constant $C'$ depending on $p,T,d',K$. According to Kolmogorov's continuity criterion, it suffices to demonstrate that for any $p\geq2$ and $M>0$, there exists a positive constant $C=C(p,T,d',K,M)$ such that
\begin{align}\label{lp1}
E|\partial_x\Phi_{s,t}(x)-\partial_x\Phi_{s',t'}(x')|^p\leq C(|s-s'|^{\frac{p}{2}}+|t-t'|^{\frac{p}{2}}+|x-x'|^p)
\end{align}
for all $(s,t,x),~(s',t',x')$ with $0\leq s < t\leq T,~0\leq s' < t'\leq T$ and $\max\{|x|,|x'|\}<M$. Assume without loss of generality that $0\leq s'<s<t<t'\leq T$.  To prove \eqref{lp1}, we conclude by the triangle inequality
\begin{align}\label{ee}
&E[|\partial_x\Phi_{s,t}(x)-\partial_x\Phi_{s',t'}(x')|^p]\\
&\leq 3^{p-1}\big(E[|\partial_x\Phi_{s,t}(x)-\partial_x\Phi_{s,t'}(x)|^p]+
E[|\partial_x\Phi_{s,t'}(x)-\partial_x\Phi_{s,t'}(x')|^p]+E[|\partial_x\Phi_{s,t'}(x')-\partial_x\Phi_{s',t'}(x')|^p]\big).\nonumber
\end{align}

 We now estimate the first term of \eqref{ee}. It follows from BDG's inequality, H\"older's inequality, Assumption \ref{con} and \eqref{BB} that
\begin{align}\label{9}
&E[|\partial_x\Phi_{s,t}(x)-\partial_x\Phi_{s,t'}(x)|^p]\\
&\leq C(p,d')\sum_{k=0}^{d'}E\Big|\int_t^{t'}\partial_x V_k\cdot\partial_x\Phi_{s,r}(x)+\widetilde{E}[\partial_\mu V_k\cdot\partial_x\widetilde{\Phi}_{s,r}(x)]\rmd
W_r^k\Big|^p\nonumber\\
&\leq C(T,p,d')\sum_{k=0}^{d'}|t-t'|^{\frac{p}{2}-1}\int_t^{t'}E\big|\partial_x V_k\cdot\partial_x\Phi_{s,r}(x)\big|^p+\big|\widetilde{E}[\partial_\mu
V_k\cdot\partial_x\widetilde{\Phi}_{s,r}(x)]\big|^p\rmd r\nonumber\\
&\leq C'(T,p,K,d')|t-t'|^{\frac{p}{2}-1}\int_t^{t'}E\big|\partial_x\Phi_{s,r}(x)\big|^p+E\big|\partial_x\Phi_{s,r}(x)\big|^p\rmd r\nonumber\\
&\leq C''(T,p,K,d')|t-t'|^{\frac{p}{2}}.\nonumber
\end{align}

Similarly, we estimate the second term and obtain
\begin{align*}
&E[|\partial_x\Phi_{s,t'}(x)-\partial_x\Phi_{s,t'}(x')|^p]-C(p,d')|x-x'|^p\\
&\leq  C(p,d')\sum_{k=0}^{d'}E\Big|\int_s^{t'}a_k(s,r,x)\partial_x\Phi_{s,r}(x)-a_k(s,r,x')\partial_x\Phi_{s,r}(x')\\
&\quad\qquad\qquad\qquad+\widetilde{E}[b_k(s,r,x)\partial_x\widetilde{\Phi}_{s,r}(x)-b_k(s,r,x')\partial_x\widetilde{\Phi}_{s,r}(x')]\rmd W_r^k\Big|^p\\
&\leq  C(T,p,d')\sum_{k=0}^{d'}\int_s^{t'}E\big|a_k(s,r,x)\partial_x\Phi_{s,r}(x)-a_k(s,r,x')\partial_x\Phi_{s,r}(x')\big|^p\\
&\quad\qquad\qquad\qquad+E\big|\widetilde{E}[b_k(s,r,x)\partial_x\widetilde{\Phi}_{s,r}(x)-b_k(s,r,x')\partial_x\widetilde{\Phi}_{s,r}(x')]\big|^p\rmd r\\
&\leq C'(T,p,d')\sum_{k=0}^{d'}\int_s^{t'}E\big|a_k(s,r,x)\big(\partial_x\Phi_{s,r}(x)-\partial_x\Phi_{s,r}(x')\big)\big|^p+E\big|\big(a_k(s,r,x)-a_k(s,r,x')\big)\partial_x\Phi_{s,r}(x')\big|^p\\
&\quad\qquad\qquad\qquad+E\widetilde{E}\big|b_k(s,r,x)\big(\partial_x\widetilde{\Phi}_{s,r}(x)-\partial_x\widetilde{\Phi}_{s,r}(x')\big)\big|^p+E\widetilde{E}\big|\big(b_k(s,r,x)-b_k(s,r,x')\big)\partial_x\widetilde{\Phi}_{s,r}(x')\big|^p\rmd r\\
&\leq  C(T,p,d',K)\sum_{k=0}^{d'}\int_s^{t'}E\big|\partial_x\Phi_{s,r}(x)-\partial_x\Phi_{s,r}(x')\big|^p
+\big(E\big|\Phi_{s,r}(x)-\Phi_{s,r}(x')\big|^{2p}\big)^{\frac{1}{2}}\big(E\big|\partial_x\Phi_{s,r}(x')\big|^{2p}\big)^{\frac{1}{2}}\rmd r\\
&\leq  C'(T,p,d',K)\sum_{k=0}^{d'}\int_s^{t'}E\big|\partial_x\Phi_{s,r}(x)-\partial_x\Phi_{s,r}(x')\big|^p\rmd r.
\end{align*}
Subsequently, Gronwall's inequality yields
\begin{align}\label{8}
E[|\partial_x\Phi_{s,t'}(x)-\partial_x\Phi_{s,t'}(x')|^p]\leq C(p,d',K,T)|x-x'|^p.
\end{align}

We next estimate the last term of \eqref{ee} and notice the fact
\begin{align*}
\partial_x\Phi_{s',t'}(x')
=\partial_x\Phi_{s',s}(x')+\sum_{k=0}^{d'}&\int_{s}^{t'} \partial_x V_k(\Phi_{s',r}(x'),P_{\Phi_{s',r}(x')})\cdot \partial_x\Phi_{s',r}(x')\\
&+\widetilde{E}\bigr[\partial_{\mu}V_k(\Phi_{s',r}(x'),P_{\Phi_{s',r}(x')},\widetilde{\Phi}_{s',r}(x'))\cdot
\partial_x\widetilde{\Phi}_{s',r}(x')\bigr]\rmd W_r^k
\end{align*}
for $s'<s<t'$. This implies
\begin{align}\label{10}
&E[|\partial_x\Phi_{s,t'}(x')-\partial_x\Phi_{s',t'}(x')|^p]\\
&\leq  C(p,d')E|\partial_x\Phi_{s',s}(x')-I|^p+ C(p,d')\sum_{k=0}^{d'}E\Big|\int_s^{t'}a_k(s,r,x')\partial_x\Phi_{s,r}(x')-a_k(s',r,x')\partial_x\Phi_{s',r}(x')\rmd W_r^k\Big|^p\nonumber\\
&\qquad+ C(p,d')\sum_{k=0}^{d'}E\Big|\int_s^{t'}\widetilde{E}[b_k(s,r,x')\partial_x\widetilde{\Phi}_{s,r}(x')-b_k(s',r,x')\partial_x\widetilde{\Phi}_{s',r}(x')]\rmd W_r^k\Big|^p.\nonumber
\end{align}
Since the third term of \eqref{10} is similar to the second term due to \eqref{p}, we focus on estimating the second term. It follows from Assumption \ref{con}, BDG's inequality and H\"older's inequality that
\begin{align*}
&\sum_{k=0}^{d'}E\Big|\int_s^{t'}a_k(s,r,x')\partial_x\Phi_{s,r}(x')-a_k(s',r,x')\partial_x\Phi_{s',r}(x')\rmd W_r^k\Big|^p\\
&\leq \sum_{k=0}^{d'}E\Big|\int_s^{t'}a_k(s,r,x')(\partial_x\Phi_{s,r}(x')-\partial_x\Phi_{s',r}(x'))+(a_k(s,r,x')-a_k(s',r,x'))\partial_x\Phi_{s',r}(x')\rmd W_r^k\Big|^p\\
&\leq C(T,p)\sum_{k=0}^{d'}\int_s^{t'}E\big|a_k(s,r,x')(\partial_x\Phi_{s,r}(x')-\partial_x\Phi_{s',r}(x'))\big|^p+E\big|\big(a_k(s,r,x')-a_k(s',r,x')\big)\partial_x\Phi_{s',r}(x')\big|^p\rmd r\\
&\leq
C(K,d',p,T)\int_s^{t'}E\big|\partial_x\Phi_{s,r}(x')-\partial_x\Phi_{s',r}(x')\big|^p+\big(E\big|\Phi_{s,r}(x')-\Phi_{s',r}(x')\big|^{2p}\big)^{\frac{1}{2}}\rmd
r\\
&\leq
C(K,d',p,T)\int_s^{t'}E\big|\partial_x\Phi_{s,r}(x')-\partial_x\Phi_{s',r}(x')\big|^p+C(K,d',p,T)(1+|x'|^{2p})^{\frac{1}{2}}|s'-s|^{\frac{p}{2}}\rmd
r\\
&\leq C(K,d',p,T)\int_s^{t'}E\big|\partial_x\Phi_{s,r}(x')-\partial_x\Phi_{s',r}(x')\big|^p\rmd r+C(K,d',p,T,M)|s'-s|^{\frac{p}{2}}.
\end{align*}
Consequently,
\begin{align*}
&E[|\partial_x\Phi_{s,t'}(x')-\partial_x\Phi_{s',t'}(x')|^p]\\
&\leq C(p,d')E|\partial_x\Phi_{s',s}(x')-I|^p+C(K,d',p,T,M)|s'-s|^{\frac{p}{2}}\\ &\qquad\qquad\qquad\qquad+C'(K,d',p,T)\int_s^{t'}E\big|\partial_x\Phi_{s,r}(x')-\partial_x\Phi_{s',r}(x')\big|^p\rmd r\\
&\leq C'(K,d',p,T,M)|s-s'|^{\frac{p}{2}}+ C''(K,d',p,T)\int_s^{t'}E\big|\partial_x\Phi_{s,r}(x')-\partial_x\Phi_{s',r}(x')\big|^p\rmd r.
\end{align*}
This implies by Gronwall's inequality that
\begin{align}\label{7}
E[|\partial_x\Phi_{s,t'}(x')-\partial_x\Phi_{s',t'}(x')|^p]\leq C''(K,d',p,T,M)|s'-s|^{\frac{p}{2}}.
\end{align}

Finally, \eqref{9}, \eqref{8} and \eqref{7} imply \eqref{lp1}. So far, we have proven that $\partial_x\Phi_{s,t}(x)$ is continuous in $(s,t,x)$ almost surely.

If $V_0,...,V_d\in \cal{C}_{b,Lip}^{n,n}(\R^d\times\cal{P}_2(\R^d),\R^d)$, we can similarly obtain that $\Phi_{s,t}(x)$ has a modification, ensuring its $n$-times differentiability with respect to $x$ and the continuity of the derivatives of all orders with respect to $(s,t,x)$.
\end{proof}

\section{The Locally Diffeomorphic Property of Solutions with respect to the Deterministic Initial Value}

In this section, we consider the following Stratonovich symmetric McKean-Vlasov SDE
\begin{align}\label{s-sde}
X_t=x+\sum_{k=0}^{d'}\int_s^t V_k(X_r,P_{X_r})\circ \rmd W_r^k,
\end{align}
where $0\leq s\leq t\leq T$ for given $T>0$. The existence and the uniqueness of the solution are guaranteed when $V_0,...,V_d$ belong to
$\cal{C}_{b,Lip}^{1,1}(\R^d\times\cal{P}_2(\R^d),\R^d)$ (see Subsection \ref{2.2}). We still denote the unique solution by $\Phi_{s,t}(x)$. Moreover, according to Theorem \ref{2}, if $V_0,...,V_d$ belong to $\cal{C}_{b,Lip}^{2,2}(\R^d\times\cal{P}_2(\R^d),\R^d)$, the mapping $x\mapsto\Phi_{s,t}(x)$ is $P$-$a.s.$ differentiable for all $s,t\in [0, T]$.

In contrast to classical SDEs, as discussed earlier, $\Phi_{s,t}(x)$ does not define a flow. Meanwhile, we will observe in this section that the globally homeomorphic property cannot be guaranteed at any time except the initial time $s$, even in cases where the solution exists globally. This highlights another distinction between McKean-Vlasov SDEs and classical SDEs.

Before discussing the locally homeomorphic property, we first demonstrate the regularity of solutions with respect to deterministic initial values for Stratonovich symmetric McKean-Vlasov SDEs, based on the results of It\^o McKean-Vlasov SDEs. Moreover, we will explore the invertibility of the Jacobian matrix $\partial_x\Phi_{s,t}(x)$.

 \begin{thm}\label{thm2}
Assume that coefficients $V_0,...,V_d$ of Stratonovich symmetric McKean-Vlasov SDE~\eqref{s-sde} belong to
$\cal{C}_{b,Lip}^{2,2}(\R^d\times\cal{P}_2(\R^d),\R^d)$. Then the Jacobian matrix $\partial_x\Phi_{s,t}(x)$ satisfies the Stratonovich symmetric McKean-Vlasov SDE
\begin{align*}
\partial_x\Phi_{s,t}(x)=I+\sum_{k=0}^{d'}&\int_s^t \partial_x V_k(\Phi_{s,r}(x),P_{\Phi_{s,r}(x)})\cdot \partial_x\Phi_{s,r}(x)\\
&\quad+\widetilde{E}\bigr[\partial_{\mu}V_k(\Phi_{s,r}(x),P_{\Phi_{s,r}(x)},\widetilde{\Phi}_{s,r}(x))\cdot
\partial_x\widetilde{\Phi}_{s,r}(x)\bigr] \circ\rmd W_r^k
\end{align*}
for any $s<t, x\in\R^d.$ Further, for any $x\in\R^d,~s\geq 0$, there exists a stopping time $\varrho(s,x)=:\varrho(x)>s$ such that $\partial_x\Phi_{s,t\wedge\varrho(x)}(x)$ is invertible for any $t\geq s$ almost surely.
\end{thm}

\begin{proof}
{\bf Step 1: We will establish that the Jacobian matrix $\partial_x\Phi_{s,t}(x)$ satisfies the above Stratonovich symmetric SDE.}
To establish this, it suffices to prove the interchangeability of integrals and the derivative $\partial_x$.
In terms of Riemann integrals, the interchangeability can be guaranteed under the condition $V_0\in\cal{C}_{b,Lip}^{2,2}(\R^d\times\cal{P}_2(\R^d),\R^d)$.
For Stratonovich symmetric integrals $(k\neq0)$, according to \cite[Proposition 2.4.3]{K}, we need to verify that $\phi_k$ and $\phi_k^1$ belong to $\mathbf{L}_{[s,T]}^{1+Lip,p}(B(0,M))$ for fixed $s\in[0,T],~M\in \mathbb{Z}^+$ and $p>d\vee2$. Here, $\phi_k(x,t)=V_k(\Phi_{s,t}(x),P_{\Phi_{s,t}(x)})$ and $\phi_k^1(x,t)=\partial_xV_k(\Phi_{s,t}(x),P_{\Phi_{s,t}(x)})V_k(\Phi_{s,t}(x),P_{\Phi_{s,t}(x)})$.

Indeed, $\phi_k\in\mathbf{L}_{[s,T]}^{1+Lip,p}(B(0,M))$ has already been proven in Theorem \ref{2}. The same result can be obtained for $\phi_k^1$ through similar estimates.

The arbitrariness of $M$ implies that for any $x\in\R^d$, the Jacobian matrix $\partial_x\Phi_{s,t}(x)$ satisfies the Stratonovich symmetric SDE
\begin{align*}
\partial_x\Phi_{s,t}(x)=I+\sum_{k=0}^{d'}&\int_s^t \partial_x V_k(\Phi_{s,r}(x),P_{\Phi_{s,r}(x)})\cdot \partial_x\Phi_{s,r}(x)\\
&\quad+\widetilde{E}\bigr[\partial_{\mu}V_k(\Phi_{s,r}(x),P_{\Phi_{s,r}(x)},\widetilde{\Phi}_{s,r}(x))\cdot
\partial_x\widetilde{\Phi}_{s,r}(x)\bigr] \circ\rmd W_r^k.
\end{align*}

{\bf Step 2: we will establish that for any $x\in\R^d$ and $s\geq 0$, there exists a stopping time $\varrho(x)$ such that $\partial_x\Phi_{s,t\wedge\varrho(x)}(x)$ is invertible for any $t>s$ almost surely.} For fixed $s\geq0$, we set $\Phi_{s,t}(x)=:\Phi_{t}(x)$ and $\widetilde{\Phi}_{s,t}(x)=:\widetilde{\Phi}_{t}(x)$. Consider the following SDE for $V_t(x)$:
\begin{align*}
V_t(x)=I-\sum_{k=0}^{d'}&\int_s^t V_r(x)\partial_x V_k(\Phi_r(x),P_{\Phi_r(x)})\partial_x\Phi_r(x)V_r(x)\\
&\quad+V_r(x)\widetilde{E}\bigr[\partial_{\mu}V_k(\Phi_r(x),P_{\Phi_r(x)},\widetilde{\Phi}_r(x))\cdot
\partial_x\widetilde{\Phi}_r(x)\bigr]V_r(x) \circ \rmd W_r^k.
\end{align*}
Due to the coefficients of this equation satisfying only the local Lipschitz condition, its unique solution may blow up. We denote the explosion time by $\varrho(s,x)=:\varrho(x).$ By It\^o's formula for Stratonovich symmetric integrals,
\begin{align*}
&~~~V_{t\wedge\varrho(x)}(x)\partial_x\Phi_{t\wedge\varrho(x)}(x)-I\\
&=\int_s^{t\wedge\varrho(x)}V_r(x)\circ \rmd\partial_x\Phi_r(x)+\circ \rmd V_r(x)(\partial_x\Phi_r(x))\\
&=-\sum_{k=0}^{d'}\int_s^{t\wedge\varrho(x)}V_r(x)\Bigr[\partial_x V_k(\Phi_r(x),P_{\Phi_r(x)})\partial_x\Phi_r(x)\\
&\quad\quad\quad+\widetilde{E}\bigr[\partial_{\mu}V_k(\Phi_r(x),P_{\Phi_r(x)},\widetilde{\Phi}_r(x))\cdot
\partial_x\widetilde{\Phi}_r(x)\bigr]\Bigr]\bigr[V_r(x)\partial_x\Phi_r(x)-I\bigr]\circ \rmd W_r^k.
\end{align*}
Applying It\^o's formula once again, we obtain
\begin{align*}
&V_{t\wedge\varrho(x)}(x)\partial_x\Phi_{t\wedge\varrho(x)}(x)-I\\
&=0\cdot\exp\Bigr\{-\sum_{k=0}^{d'}\int_s^{t\wedge\varrho(x)}V_r(x)\Bigr[\partial_x V_k(\Phi_r(x),P_{\Phi_r(x)})\partial_x\Phi_r(x)\\
&\quad\quad\quad+\widetilde{E}\bigr[\partial_{\mu}V_k(\Phi_r(x),P_{\Phi_r(x)},\widetilde{\Phi}_r(x))\cdot
\partial_x\widetilde{\Phi}_r(x)\bigr]\Bigr]\circ \rmd W_r^k\Bigr\}=0.
\end{align*}
Therefore we have $V_{t\wedge\varrho(x)}(x)\partial_x\Phi_{t\wedge\varrho(x)}(x)=I$, which implies that the Jacobian matrix $\partial_x\Phi_{t\wedge\varrho(x)}(x)$ is invertible for any $t>s$ almost surely.
\end{proof}

\begin{rem}\label{4.2}
\begin{enumerate}
\item By Theorem \ref{thm2}, we observe that the first time $t$ such that $\partial_x\Phi_{s,t}(x)$ is not invertible depends on $x$. In the subsequent analysis, we will discover that this leads to the local existence of the inverse function of $\Phi_{s,t}(\omega)$ for almost all $\omega$.
\item There may exist another stopping time $\varrho'(x)$ such that $\partial_x\Phi_{s,t}(x)$ is invertible on $(\varrho(x),\varrho'(x))$. However, as long as the stopping time $\varrho'(x)$ cannot attain the positive infinity, it does not make much difference whether we consider $\varrho(x)$ or $\varrho'(x)$. Therefore, when demonstrating the homeomorphic property below, we only focus on the first time that satisfies certain properties.
\item If $V_0,...,V_d$ belong to $\cal{C}_{b,Lip}^{k,k}(\R^d\times\cal{P}_2(\R^d),\R^d)(k\geq2)$, it holds by Theorem \ref{2} that $\Phi_{s,t}(x)$ has a modification such that it is $C^{k-1}$ with respect to $x$ and the derivatives of all order are continuous in $(s,t,x)$ almost surely.
\end{enumerate}
\end{rem}

Let $\mathcal{M}$ be a given compact set in $\R^d$ and $\vartheta_m$ be the first time $t$ such that $\sup_{x\in\mathcal{M}}|\partial_x\Phi_{t}(x)^{-1}|>m$. It holds that $\lim_{m\rightarrow\infty}\vartheta_m=\inf_{x\in\mathcal{M}}\varrho(x)$. For given $T>0$, consider the following Stratonovich symmetric SDE
\begin{equation}\label{3}
\left\{
\begin{aligned}
\circ \rmd X^m_t&=-\sum_{k=0}^{d'}\partial_x\Phi_{t\wedge\vartheta_m}(X^m_t)^{-1}
V_k(\Phi_{t}(X^m_t),\widetilde{P}_{\widetilde{\Phi}_{t}(X^m_t)})\circ \rmd W_t^k,\\
X^m_s&=x\in\mathcal{M},\\
\end{aligned}
\right.
\end{equation}
where $t\in [s,T]$ and $\widetilde{P}_{\widetilde{\Phi}_{t}(\cdot)}:\R^d\rightarrow\cal{P}_2(\R^d)$. Here, to avoid confusion, we use a copy of $\Phi_{t}(x)$, denoted as $\widetilde{\Phi}_{t}(x)$, instead of $\Phi_{t}(x)$. We will discuss  the existence and uniqueness of solutions to equation \eqref{3} under $V_0,...,V_d\in\cal{C}_{b,Lip}^{4,4}(\R^d\times\cal{P}_2(\R^d),\R^d)$. For any $n>0$, let $\tau_n$ be the first time $t$ such that
\begin{align}\label{st}
\sup_{x\in\mathcal{M}}\max\{|\Phi_t(x)|,|\partial_x\Phi_{t}(x)|,|\nabla^2\Phi_t(x)|,|\nabla^3\Phi_t(x)|\}>n,
\end{align}
where $\nabla^2\Phi_t(x)=(\partial_{x_ix_j}\Phi_t(x))_{1\leq i,j\leq d}$ and $\nabla^3\Phi_t(x)=(\partial_{x_ix_jx_l}\Phi_t(x))_{1\leq i,j,l\leq d}$.
$\{\tau_n\}$ are stopping times and $\lim_{n\rightarrow\infty}\tau_n=\infty$ because the above four items exist globally according to (3) of Remark \ref{4.2}. Consider
$$\circ \rmd X^{m,n}_t=-\sum_{k=0}^{d'}\partial_x\Phi_{t\wedge\vartheta_m}(X^{m,n}_t)^{-1}
V_k(\Phi_{t\wedge{\tau_n}}(X^{m,n}_t),\widetilde{P}_{\widetilde{\Phi}_t(X^{m,n}_t)})\circ \rmd W_t^k.$$
We denote random coefficients $-\partial_x\Phi_{t\wedge\vartheta_m}(x)^{-1}V_k(\Phi_{t\wedge{\tau_n}}(x),\widetilde{P}_{\widetilde{\Phi}_t(x)})$ by $\cal{V}^{m,n}_k(x,t),~k=0,...,d'$. For the convenience of calculation, we transform the equation into the classical It\^o SDE:
\begin{align}\label{1}
X^{m,n}_t=x+\sum_{k=1}^{d'}\int_s^t\cal{V}^{m,n}_k(X^{m,n}_r,r)\rmd W_r^k+\int_s^t\widetilde{\cal{V}}^{m,n}_0(X^{m,n}_r,r)\rmd r,
\end{align}
where $$\widetilde{\cal{V}}^{m,n}_0(x,r)=\cal{V}^{m,n}_0(x,r)+\frac{1}{2}\sum_{k=1}^{d'}\partial_x\cal{V}^{m,n}_k(x,r)\cal{V}^{m,n}_k(x,r).$$

\begin{lem}
When $V_0,...,V_d$ belong to $\cal{C}_{b,Lip}^{4,4}(\R^d\times\cal{P}_2(\R^d),\R^d)$, the solution of equation \eqref{1} exists uniquely.
\end{lem}

\begin{proof}
Without loss of generality, we suppose that $\cal{M}$ is the closure of $B(0,M)$. We will demonstrate that $\cal{\widetilde{V}}^{m,n}_0,\cal{V}^{m,n}_k(x,r)$ with $k=1,...,d'$ satisfy the global linear growth and global Lipschitz conditions. In fact, it holds by \eqref{B} that for all $x\in\cal{M}$,
\begin{align}\label{6} |\cal{V}^{m,n}_k(x,r)|&=|\partial_x\Phi_{r\wedge\vartheta_m}(x)^{-1}V_k(\Phi_{r\wedge{\tau_n}}(x),\widetilde{P}_{\widetilde{\Phi}_r(x)})|\\
	&\leq m\big(|V_k(\Phi_{r\wedge{\tau_n}}(x),\widetilde{P}_{\widetilde{\Phi}_r(x)})-V_k(0,\delta_0)|+|V_k(0,\delta_0)|\big)\nonumber\\
	&\leq m\big(|\Phi_{r\wedge{\tau_n}}(x)|+(E|\Phi_r(x)|^2)^{\frac{1}{2}}+|V_k(0,\delta_0)|\big)\nonumber\\
	&\leq m\big(|\Phi_{r\wedge{\tau_n}}(x)|+(C(1+|x|^2))^{\frac{1}{2}}+|V_k(0,\delta_0)|\big)\nonumber\\
	&\leq C(m,n,T,k,K,d')(1+|x|),\nonumber
\end{align}
which implies that $\cal{V}^{m,n}_k(x,r)$ satisfies the global linear growth condition for $k=0,...,d'$.

After a simple calculation by the definition of $\tau_n$, there exists a constant $C(K,T,M,n,m)$ such that for any $(x,r)\in B(0,M)\times[s,T]$,
\begin{align}\label{4}
|\partial_x\cal{V}^{m,n}_k(x,r)|\leq C(K,T,M,n,m),
\end{align}
\begin{align}\label{5}
|\nabla^2\cal{V}^{m,n}_k(x,r)|\leq C(K,T,M,n,m).
\end{align}
Therefore, it follows from \eqref{6} and \eqref{4} that
\begin{align*}
\Big|\sum_{k=1}^{d'}\partial_x\cal{V}^{m,n}_k(x,r)\cal{V}^{m,n}_k(x,r)\Big|
&\leq\sum_{k=1}^{d'}|\cal{V}^{m,n}_k(x,r)\|\partial_x\cal{V}^{m,n}_k(x,r)|\\
&\leq C(K,T,M,n,d',m)(1+|x|).
\end{align*}
By combining \eqref{6}, we can conclude that $\cal{\widetilde{V}}^{m,n}_0$ satisfies the global linear growth condition.

Next, we will show that the coefficients satisfy the global Lipschitz condition. In fact, for any $x,y\in\cal{M}$ and $\tau\in[0,1]$, letting $z_\tau:=y+\tau(x-y)\in\cal{M}$, we obtain
\begin{align*}
&|\cal{V}^{m,n}_k(x,r)-\cal{V}^{m,n}_k(y,r)|\\
&=|\partial_x\Phi_{r\wedge\vartheta_m}(x)^{-1}V_k(\Phi_{r\wedge{\tau^n}}(x),\widetilde{P}_{\widetilde{\Phi}_r(x)})
	-\partial_x\Phi_{r\wedge\vartheta_m}(y)^{-1}V_k(\Phi_{r\wedge{\tau^n}}(y),\widetilde{P}_{\widetilde{\Phi}_r(y)})|\\
&=|\int_0^1\frac{\rmd}{\rmd
	\tau}[\partial_x\Phi_{r\wedge\vartheta_m}(z_\tau)^{-1}V_k(\Phi_{r\wedge{\tau^n}}(z_\tau),\widetilde{P}_{\widetilde{\Phi}_r(z_\tau)})]\rmd
	\tau|\\
&\leq|x-y|\Bigr(\int_0^1|\partial_x\Phi_{r\wedge\vartheta_m}(z_\tau)^{-1}\partial_x V_k\cdot\partial_x\Phi_{r\wedge{\tau^n}}(z_\tau)|\rmd \tau+\int_0^1|\partial_x\Phi_{r\wedge\vartheta_m}(z_\tau)^{-1}\widetilde{E}[\partial_\mu V_k\cdot\partial_x\widetilde{\Phi}_r(z_\tau)]|\rmd \tau\\ &\qquad\qquad+\int_0^1|\partial_x\Phi_{r\wedge\vartheta_m}(z_\tau)^{-1}\nabla^2\Phi_{r\wedge\vartheta_m}(z_\tau)\partial_x\Phi_{r\wedge\vartheta_m}(z_\tau)^{-1}
	V_k(\Phi_{r\wedge{\tau^n}}(z_\tau),\widetilde{P}_{\widetilde{\Phi}_r(z_\tau)})|\rmd \tau\Bigr)\\
&\leq|x-y|\Big(nmK+mK\int_0^1\widetilde{E}|\partial_x\widetilde{\Phi}_r(z_\tau)|\rmd
	\tau+m^2C(n,M,T,d',K)\int_0^1|\nabla^2\Phi_{r\wedge\vartheta_m}(z_\tau)|\rmd \tau\Big)\\
&\leq C(K,T,M,n,m,d')|x-y|,
\end{align*}
which implies that $\cal{V}^{m,n}_k$ satisfies the global Lipschitz condition for $k=0,...,d'$.
It follows from \eqref{5} that
\begin{align*}
|\partial_x\cal{V}^{m,n}_k(x,r)-\partial_x\cal{V}^{m,n}_k(y,r)|&\leq\Big|\int_0^1\nabla^2\cal{V}^{m,n}_k(z_\tau,r)\rmd\tau\Big|~|x-y|\\
	&\leq \int_0^1|\nabla^2\cal{V}^{m,n}_k(z_\tau,r)|\rmd \tau\cdot|x-y|\\
	&\leq C(K,T,M,n,m)|x-y|.
\end{align*}
Further,
\begin{align*}
	&|\sum_{k=1}^{d'}\partial_x\cal{V}^{m,n}_k(x,r)\cal{V}^{m,n}_k(x,r)-\sum_{k=1}^{d'}\partial_x\cal{V}^{m,n}_k(y,r)\cal{V}^{m,n}_k(y,r)|\\
	&\leq\sum_{k=1}^{d'}|\cal{V}^{m,n}_k(x,r)\|\partial_x\cal{V}^{m,n}_k(x,r)-\partial_x\cal{V}^{m,n}_k(y,r)|
	+|\cal{V}^{m,n}_k(x,r)-\cal{V}^{m,n}_k(y,r)\|\partial_x\cal{V}^{m,n}_k(y,r)|\\
	&\leq C(K,T,M,n,m)|x-y|,
\end{align*}
which implies that $\cal{\widetilde{V}}^{m,n}_0$ satisfies the global Lipschitz condition. By \cite[Theorem~5.5.1]{AF}, the solution of equation \eqref{1} exists uniquely.
\end{proof}

For given $n,~m\geq 1$, let $X^{m,n,x}_t$ be the unique solution to \eqref{1} starting from $x\in\mathcal{M}$ at $s$. By the path-wise uniqueness of the solution, it holds that $X^{m,n,x}_t=X^{m,n+1,x}_t$ if $t <\tau_n$. Therefore there exists $X^{m,x}_t$ such that $X^{m,x}_t= X^{m,n,x}_t$ holds almost surely if $t < \tau_n$. Then $X^{m,x}_t$ is the unique solution of equation \eqref{3}. Let $\overline{\tau}_m(x)$ be the first time $t$ such that $\partial_x\Phi_{t}(X^{m,x}_t)^{-1}$ is not well-defined. Then it holds that
\begin{align*}
X^{m,x}_t=x-\sum_{k=0}^{d'}\int_s^t\partial_x\Phi_{r}(X^{m,x}_r)^{-1}
V_k(\Phi_{r}(X^{m,x}_r),\widetilde{P}_{\widetilde{\Phi}_{r}(X^{m,x}_r)})\circ \rmd W_r^k
\end{align*}
when $t\in[s,\overline{\tau}_m(x))$. From the definition of $\vartheta_m$, it holds that $s\leq\overline{\tau}_m(x)\leq\vartheta_m$ for any $x\in\mathcal{M}$. Hence, we obtain a sequence of stopping times $\{\overline{\tau}_m(x)\}_{m=1}^{+\infty}$ satisfying
$$s\leq\overline{\tau}_1(x)\leq...\leq\overline{\tau}_m(x)\leq\overline{\tau}_{m+1}(x)\leq...\leq\lim_{m\rightarrow\infty}\vartheta_m
=\inf_{x\in\mathcal{M}}\varrho(x),$$
and $\lim_{m\rightarrow\infty}\overline{\tau}_m(x)=:\overline{\tau}(x)$ exists for any $x\in\mathcal{M}$. At the end of this paper, we will provide an example to illustrate that $\varrho(x)$ in general cannot take on the value $+\infty$, and as a consequence, neither can $\overline{\tau}(x)$. By applying a similar approach as we did for `$n$' to handle `$m$', we obtain $X_t^x$ satisfying $X^{x}_t= X^{m,x}_t$ almost surely if $t<\overline{\tau}_m(x)$. We denote $X_t^x$ by $\Psi_{s,t}(x)$ satisfying
\begin{align}\label{11}
\Psi_{s,t\wedge\overline{\tau}(x)}(x)=x-\sum_{k=0}^{d'}\int_s^{t\wedge\overline{\tau}(x)}\partial_x\Phi_{s,r}(\Psi_{s,r}(x))^{-1}
V_k(\Phi_{s,r}(\Psi_{s,r}(x)),\widetilde{P}_{\widetilde{\Phi}_{s,r}(\Psi_{s,r}(x))})\circ \rmd W_r^k.
\end{align}
Due to the arbitrariness of $\mathcal{M}$, equation \eqref{11} holds for any $x\in\R^d$.

We are now ready to state and prove the main result of this section.
\begin{thm}\label{4.4}
For any $x\in\R^d$ and $s\geq 0$, there exists a stopping time $\tau(s,x)>s$ such that
\begin{align*}
\Phi_{s,t\wedge\tau(s,x)}(\Psi_{s,t\wedge\tau(s,x)}(x))=\Psi_{s,t\wedge\tau(s,x)}(\Phi_{s,t\wedge\tau(s,x)}(x))=x,~t\geq s
\end{align*}
holds almost surely.
\end{thm}

\begin{proof}
For fixed $s\geq0$, we set $\Phi_{s,t}(x)=:\Phi_{t}(x),~\Psi_{s,t}(x)=:\Psi_{t}(x)$ and omit $\sum_{k=0}^{d'}$ in the following proof. By a similar computation as in Step 2 of Theorem \ref{2}, we establish that $\Phi_{t}(x)$ and $\Psi_{t}(x)$ satisfy the conditions of Proposition \ref{GIF}. Therefore, applying Proposition \ref{GIF} on $[s,\overline{\tau}(x))$, we have
\begin{align*}
&\circ \rmd\Phi_t(\Psi_t(x))\\
&=V_k(\Phi_t(\Psi_t(x)),\widetilde{P}_{\widetilde{\Phi}_t(\Psi_t(x))})\circ \rmd W_t^k
\!+\!\partial_x\Phi_t(\Psi_t(x))\circ  \rmd\Psi_t(x)\\
&=V_k(\Phi_t(\Psi_t(x)),\widetilde{P}_{\widetilde{\Phi}_t(\Psi_t(x))})\circ \rmd W_t^k\\
&\quad-\partial_x\Phi_t(\Psi_t(x))\cdot\partial_x\Phi_t(\Psi_t(x))^{-1}
V_k(\Phi_t(\Psi_t(x)),\widetilde{P}_{\widetilde{\Phi}_t(\Psi_t(x))})\circ \rmd W_t^k\\
&=0,
\end{align*}
which implies that $\Phi_{t\wedge\overline{\tau}(x)}(\Psi_{t\wedge\overline{\tau}(x)}(x))=x$. Further, when $t\in[s,\overline{\tau}(x))$, $\partial_x\Psi_{t}(x)=\partial_x\Phi_{t}(\Psi_{t}(x))^{-1}$ holds, and consequently,
$$\circ\rmd\Psi_t(x)=-\partial_x\Psi_{t}(x)V_k(\Phi_t(\Psi_t(x)),\widetilde{P}_{\widetilde{\Phi}_t(\Psi_t(x))})\circ \rmd W_t^k.$$
We denote by $\overline{\tau'}(x)$ the first time $t$ such that $\partial_x\Psi_{t}(\Phi_t(x))$ is not well-defined. Applying Proposition~\ref{GIF} on $[s,\overline{\tau'}(x))$, we obtain
\begin{align*}
&\circ \rmd\Psi_t(\Phi_t(x))\\
&=\!-\!\partial_x\Psi_t(\Phi_t(x))V_k(\Phi_t(x),\widetilde{P}_{\widetilde{\Phi}_t
(\Psi_t(\Phi_t(x)))})
\circ \rmd W_t^k \!+\!\partial_x\Psi_t(\Phi_t(x))\circ \rmd\Phi_t(x)\\
&=\!-\!\partial_x\Psi_t(\Phi_t(x))\Big[V_k(\Phi_t(x),\widetilde{P}_{\widetilde{\Phi}_t
(\Psi_t(\Phi_t(x)))})-V_k(\Phi_t(x),\widetilde{P}_{\widetilde{\Phi}_t(x)})\Big]\circ \rmd W_t^k\\
&=\!-\!\partial_x\Psi_t(\Phi_t(x))\int_0^1\frac{\partial}{\partial \tau}
\Big[V_k\big(\Phi_t(x),\widetilde{P}_{\widetilde{\Phi}_{t}\big(x+\tau(\Psi_t(\Phi_t(x))-x)\big)}\big)\Big]\rmd \tau\circ \rmd W_t^k\\
&=\!-\!\partial_x\Psi_t(\Phi_t(x))~\Pi(x,t)~(\Psi_t(\Phi_t(x))-x)\circ \rmd W_t^k,
\end{align*}
where
\begin{align*}
&\Pi(x,t)\\
&=\int_0^1\widetilde{E}\Big[\partial_\mu
V_k\big(\Phi_t(x),\widetilde{P}_{\widetilde{\Phi}_{t}\big(x+\tau(\Psi_t(\Phi_t(x))-x)\big)},
\widetilde{\Phi}_{t}\big(x\!+\!\tau(\Psi_t(\Phi_t(x))\!-\!x)\big)\big)\!\cdot\!\partial_x\widetilde{\Phi}_{t}\big(x\!+\!\tau(\Psi_t(\Phi_t(x))\!-\!x)\big)\big)\Big]\rmd \tau.
\end{align*}
Further,
\begin{align*}
&\Psi_{t\wedge\overline{\tau'}(x)}(\Phi_{t\wedge\overline{\tau'}(x)}(x))-x\\
&=-\int_0^{t\wedge\overline{\tau'}(x)}\partial_x\Psi_r(\Phi_r(x))\Pi(x,r)(\Psi_r(\Phi_r(x))-x)\circ \rmd W_r^k.
\end{align*}
By It\^o's formula,
\begin{align*}
&\Psi_{t\wedge\overline{\tau'}(x)}(\Phi_{t\wedge\overline{\tau'}(x)}(x))-x\\
&=0\exp\Big\{-\int_0^{t\wedge\overline{\tau'}(x)}\partial_x\Psi_r(\Phi_r(x))\Pi(x,r)\circ \rmd W_r^k\Big\}\\
&=0.
\end{align*}
In fact, $\overline{\tau}(x)$ and $\overline{\tau'}(x)$ depend on $s$ and are strictly greater than $s$. We set $\tau(s,x):=\min\{\overline{\tau}(x),\overline{\tau'}(x)\}$ which is what we need. The proof is complete.
\end{proof}


For $t\geq s$, we denote the set $\{x:\tau(s,x)>t\}$ by $D_{s,t}$. Since $\partial_x\Phi_{s,t}(x)$ is continuous with respect to $x$ almost surely,
$\tau(s,x)$ is also continuous with respect to $x$ almost surely. This implies that the set $D_{s,t}$ is an open set almost surely.
According to Theorem \ref{4.4},  for almost all $\omega\in\Omega$,  $\Phi_{s,t}(\cdot,\omega)|_{D_{s,t}(\omega)}=\Phi_{s,t\wedge\tau(s,\cdot,\omega)}(\cdot,\omega)|_{D_{s,t}(\omega)}$ is a homeomorphism from $D_{s,t}(\omega)$ to its image set $R_{s,t}(\omega):=\{\Phi_{s,t}(x,\omega):x\in D_{s,t}(\omega)\}.$
Meanwhile, for any $x\in D_{s,t}(\omega)$, the Jacobian matrix $\partial_x\Phi_{s,t}(x,\omega)$ is invertible, which implies that the inverse mapping $\Psi_{s,t}(\omega)$ is again of $C^k$-class if $\Phi_{s,t}(\omega)$ belongs to $C^k$-class by the implicit function theorem.
 Finally, the conclusion is summarized as follows:
\begin{thm}\label{thm}
Suppose that $V_0,...,V_d$ belong to $\cal{C}_{b,Lip}^{k,k}(\R^d\times\cal{P}_2(\R^d),\R^d)(k\geq4)$. Then the mapping $\Phi_{s,t}:D_{s,t}\rightarrow R_{s,t}$ is a $C^{k-1}$-diffeomorphism for any $s<t$ almost surely.
\end{thm}

\begin{rem}
\begin{enumerate}
  \item $\Phi_{s,s}(\omega)$ is the identical mapping for almost all $\omega\in\Omega$.
  \item If $\tau(s,x)$ is independent of $x$, denoted by $\tau(s)$, $\Phi_{s,t}(\omega)$ is globally diffeomorphic as $t<\tau(s,\omega)$ for almost all $\omega\in\Omega$; see the special case of Example \ref{ex}.
  \item $D_{s,t}(\omega)$ is decreasing with respect to $t$ for almost all $\omega\in\Omega$.
  \item In contrast to classical SDEs, whose coefficients satisfy the globally Lipschitz condition, the solution of McKean-Vlasov SDEs in general cannot be globally homeomorphic at any time except the initial time $s$.
\end{enumerate}
\end{rem}

We present an example to further illustrate the conclusions of this paper.

\begin{exam}\label{ex}
Consider the following $1$-dimensional McKean-Vlasov SDE
\begin{equation}
\left\{
\begin{aligned}
\nonumber
\rmd X_t&=\int_{\mathbb{R}}\big(f(X_t)-z\big) P_{X_t}(\rmd z) ~\rmd t+\int_{\mathbb{R}}\big(X_t-z\big) P_{X_t}(\rmd z) ~\rmd W_t,\\
X_s&=x,\\
\end{aligned}
\right.
\end{equation}
where $W_t$ is a $1$-dimensional Brownian motion.  Assume that the function $f:\R\rightarrow\R$ is $C^\infty$ and satisfies the global Lipschitz condition, which ensures that the coefficients of the equation are globally Lipschitz continuous. Therefore, the unique solution of the equation exists globally, still denoted by $\Phi_{s,t}(x)$.

Set $F(P_v):=\int_{\R}z~P_v(\rmd z)$, $P_v\in\cal{P}_2(\R)$. Then $F$ is Lions differentiable in $\cal{P}_2(\R)$ and the Lions derivative of $F$ at $P_{v_0}\in\cal{P}_2(\R)$ satisfies $\partial_\mu F(P_{v_0},y)=1$ for any $y\in\R$; see \cite[Example~2.2]{BLPR} for details.
Therefore, Theorem \ref{2} implies that the Jacobian matrix of $\Phi_{s,t}(x)$ satisfies
\begin{align*}
\partial_x\Phi_{s,t}(x)=1+\int_s^t\partial_xf(\Phi_{s,r}(x))\partial_x\Phi_{s,r}(x)-E\big[\partial_x\Phi_{s,r}(x)\big]\rmd r+
\int_s^t\partial_x\Phi_{s,r}(x)-E\big[\partial_x\Phi_{s,r}(x)\big]\rmd W_r.
\end{align*}
By \cite[Theorem 8.4.2]{ARn},
\begin{align*}
&\partial_x\Phi_{s,t}(x)\\
&=e^{\int_s^t\partial_xf(\Phi_{s,r}(x))\rmd r-\frac{1}{2}(t-s)+W_t-W_s}
\Big(1-\int_s^te^{-\int_s^r\partial_xf(\Phi_{s,u}(x))\rmd u+\frac{1}{2}(r-s)-W_r+W_s}E\big[\partial_x\Phi_{s,r}(x)\big]\rmd W_r\Big)\\
&=e^{\int_s^t\partial_xf(\Phi_{s,r}(x))\rmd r-\frac{1}{2}t+W_t}
\Big(e^{\frac{1}{2}s-W_s}-\int_s^te^{-\int_s^r\partial_xf(\Phi_{s,u}(x))\rmd u+\frac{1}{2}r-W_r}E\big[\partial_x\Phi_{s,r}(x)\big]\rmd W_r\Big).
\end{align*}
Set
$$f_s(t,x):=\int_s^te^{-\int_s^r\partial_xf(\Phi_{s,u}(x))\rmd u+\frac{1}{2}r-W_r}E\big[\partial_x\Phi_{s,r}(x)\big]\rmd W_r.$$
Note that $f_s(\cdot,x)$ is continuous almost surely and $f_s(s,x)=0<e^{\frac{1}{2}s-W_s}$.
For given $x\in\R$ and $s\geq0$, denote by $\varrho(s,x)$ the first time $t$ such that
$$e^{\frac{1}{2}s-W_s}\leq f_s(t,x).$$
We consider $\{x:\varrho(s,x)>t\}=:D_{s,t}$. Then for almost all $\omega$, the sign of $\partial_x\Phi_{s,t}(x,\omega)$ is strictly greater than zero for any $x\in D_{s,t}(\omega)$. However, outside of $D_{s,t}(\omega)$, the sign becomes indefinite due to It\^o's integral. Consequently, for any $t>s$, $\Phi_{s,t}(\omega)$ is a local homeomorphism from $D_{s,t}(\omega)$ to its image set for almost all $\omega$.

In particular, if $f:\R\rightarrow\R$ is the identical mapping, through straightforward calculations, we obtain
\begin{align*}
\partial_x\Phi_{s,t}(x)&=e^{\frac{1}{2}(t-s)+W_t-W_s}\Big(1-\int_s^te^{-\frac{1}{2}(r-s)-W_r+W_s}E\big[\partial_x\Phi_{s,r}(x)\big]\rmd W_r\Big)\\
&=e^{\frac{1}{2}t+W_t}\Big(e^{-\frac{1}{2}s-W_s}-\int_s^te^{-\frac{1}{2}r-W_r}\rmd W_r\Big),
\end{align*}
where the second equality holds because $E\big[\partial_x\Phi_{s,t}(x)\big]\equiv1$. For given $s\geq0$, we denote by $\varrho(s)$ the first time $t$ such that
$$e^{-\frac{1}{2}s-W_s}\leq\int_s^te^{-\frac{1}{2}r-W_r}\rmd W_r=:f_s(t).$$
Obviously, the stopping time $\varrho(s)$ is independent of $x$. This implies that for almost all $\omega\in\Omega$, $\partial_x\Phi_{s,t}(x,\omega)$ is strictly greater than zero for any $x\in\mathbb{R}$, which guarantees that $\Phi_{s,t}(\omega)$ is a homeomorphism from $\mathbb{R}$ to itself when $t<\varrho(s,\omega)$ for given $s>0$. However, when $t\geq\varrho(s,\omega)$, the sign of $\partial_x\Phi_{s,t}(x,\omega)$ is indefinite for any $x\in\mathbb{R}$ due to It\^o's integral, making it difficult to determine the homeomorphic property over longer time intervals.
\end{exam}

\begin{rem}
\begin{enumerate}
  \item This example illustrates that $\tau(s,x)$ in Theorem \ref{thm} in general cannot take on the value $+\infty$, as $\varrho(x,s)$ is greater than $\tau(s,x)$.
  \item This example indicates that the globally homeomorphic property cannot be guaranteed at any time except the initial time, even when the global Lipschitz condition is satisfied. Specifically, if the stopping time $\varrho(x,s)$ is independent of $x$, then the globally homeomorphic property can only be guaranteed up to the stopping time, rather than at arbitrary times. From this perspective, we can predict that the dynamic behavior of McKean-Vlasov SDEs is more complex than that of classical SDEs.
\end{enumerate}
\end{rem}


\begin{thebibliography}{xx}
\bibitem{ARn}
L. Arnold, \emph{Stochastic Differential Equations: Theory and Applications}. Wiley-Interscience [John Wiley \& Sons], New York-London-Sydney, 1974.
\bibitem{BM2}
K. Bahlali, M. A. Mezerdi and B. Mezerdi, Stability of McKean-Vlasov stochastic differential equations and applications.
\emph{Stoch. Dyn.}, \textbf{20} (2020), 2050007, 19.
\bibitem{PB}
P. Baxendale, Wiener processes on manifolds of maps. \emph{Proc. Roy. Soc. Edinburgh Sect. A}, \textbf{87} (1980/81), 127--152.
\bibitem{B}
J. M. Bismut, A generalized formula of It\^o and some other properties of stochastic flows. \emph{Z. Wahrsch. Verw. Gebiete},
\textbf{55} (1981), 331--350.
\bibitem{BLPR}
R. Buckdahn, J. Li, S. Peng and C. Rainer, Mean-field stochastic differential equations and associated PDEs. \emph{Ann. Probab.}, \textbf{45} (2017), 824--878.
\bibitem{CP}
P. Cardaliaguet, \emph{Notes on Mean Field Games}. From P.-L. Lions lectures at College de France (2010).
\bibitem{C}
R. Carmona, \emph{Lectures on BSDEs, Stochastic Control, and Stochastic Differential Games with Financial Applications}. SIAM, Philadelphia, 2016.
\bibitem{CD1}
R. Carmona and F. Delarue, Forward-backward stochastic differential equations and controlled McKean-Vlasov
dynamics. \emph{Ann. Probab.}, \textbf{43} (2015), 2647--2700.
\bibitem{CM}
D. Crisan and E. McMurray, Smoothing properties of McKean-Vlasov SDEs. \emph{Probab. Theory Related Fields}, \textbf{171} (2018), 97--148.
\bibitem{F}
G. dos Reis, W. Salkeld and J. Tugaut, Freidlin-Wentzell LDP in path space for McKean-Vlasov equations and the
functional iterated logarithm law. \emph{Ann. Appl. Probab.}, \textbf{29} (2019), 1487--1540.
\bibitem{KDE}
K. D. Elworthy, Stochastic dynamical systems and their flows. In \emph{Stochastic Analysis (Proc. Internat. Conf., Northwestern Univ., Evanston, Ill., 1978)}, 79--95. Academic Press, New York-London, 1978.
\bibitem{AF}
A. Friedman, \emph{Stochastic Differential Equations and Applications}. Dover Publications, Inc., Mineola, NY, 2006. Two
volumes bound as one, Reprint of the 1975 and 1976 original published in two volumes.
\bibitem{IW}
N. Ikeda and S. Watanabe, \emph{Stochastic Differential Equations and Diffusion Processes}, volume \textbf{24} of \emph{North-Holland
Mathematical Library}. North-Holland, Amsterdam, second edition, 1989.

\bibitem{KAC}
M. Kac, Foundations of kinetic theory. In \emph{Proceedings of the Third Berkeley Symposium on Mathematical Statistics and Probability, 1954--1955, vol. III}, 171--197. University of California Press, Berkeley-Los Angeles, Calif., 1956.
\bibitem{K1}
H. Kunita, Stochastic differential equations and stochastic flows of homeomorphisms. In \emph{Stochastic Analysis and
Applications}, volume \textbf{7} of \emph{Adv. Probab. Related Topics}, 269--291. Dekker, New York, 1984.
\bibitem{K}
H. Kunita, \emph{Stochastic Flows and Jump-Diffusions}, volume \textbf{92} of \emph{Probability Theory and Stochastic Modelling}.
Springer, Singapore, 2019.
\bibitem{LM}
J. Li and H. Min, Weak solutions of mean-field stochastic differential equations and application to zero-sum stochastic differential games. \emph{SIAM J. Control Optim.}, \textbf{54} (2016), 1826--1858.
\bibitem{Ma}
Z. Liu and J. Ma, Existence, uniqueness and exponential ergodicity under Lyapunov conditions for McKean-Vlasov
SDEs with Markovian switching. \emph{J. Differential Equations}, \textbf{337} (2022), 138--167.
\bibitem{M}
P. Malliavin, Stochastic calculus of variation and hypoelliptic operators. In \emph{Proceedings of the International Symposium on Stochastic Differential Equations (Res. Inst. Math. Sci., Kyoto Univ., Kyoto, 1976)}, 195--263. Wiley, New York-Chichester-Brisbane, 1978.
\bibitem{MR}
S. M\'el\'eard and S. Roelly, A propagation of chaos result for a system of particles with moderate interaction. \emph{Stochastic Process. Appl.}, \textbf{26} (1987), 317--332.
\bibitem{Me}
P. A. Meyer, Correction: ``Flow of a stochastic differential equation (following Malliavin, Bismut, Kunita)". In
\emph{Seminar on Probability, XVI}, volume \textbf{920} of \emph{Lecture Notes in Math.}, p. 623. Springer, Berlin-New York, 1982.
\bibitem{MV}
Y. Mishura and A. Veretennikov, Existence and uniqueness theorems for solutions of McKean-Vlasov stochastic equations. \emph{Theory Probab. Math. Statist.}, No. 103 (2020), 59--101.
\bibitem{O}
K. Oelschl\"ager, A law of large numbers for moderately interacting diffusion processes. \emph{Z. Wahrsch. Verw. Gebiete}, \textbf{69} (1985), 279--322.
\bibitem{RZ}
M. R\"ockner and X. Zhang, Well-posedness of distribution dependent SDEs with singular drifts. \emph{Bernoulli}, \textbf{27} (2021), 1131--1158.
\bibitem{W}
F. Wang, Distribution dependent SDEs for Landau type equations. \emph{Stochastic Process. Appl.}, \textbf{128} (2018), 595--621.
\end{thebibliography}
\end{document}